\documentclass[reqno,10pt]{amsart}
\usepackage{amssymb}
\usepackage{latexsym}
\usepackage{hyperref}
\usepackage{amsmath}
\usepackage{amscd}
\usepackage{color}
\usepackage{enumerate}
\usepackage{amsfonts}
\usepackage{graphicx}
\usepackage{mathrsfs}
\usepackage{pifont}
\usepackage{setspace}

\usepackage{tabu}
\newcommand{\eps}{{\varepsilon}}
\theoremstyle{plain}
\newtheorem{theorem}{Theorem}
\newtheorem{proposition}[theorem]{Proposition}
\newtheorem{lemma}[theorem]{Lemma}
\newtheorem{corollary}[theorem]{Corollary}

\theoremstyle{definition}
\newtheorem{definition}[theorem]{Definition}
\newtheorem{remark}[theorem]{Remark}

\numberwithin{equation}{section}
\numberwithin{theorem}{section}
\let\Re=\undefined\DeclareMathOperator*{\Re}{Re}
\let\Im=\undefined\DeclareMathOperator*{\Im}{Im}

\def\ge{\geqslant}
\def\le{\leqslant}
\def\geq{\geqslant}
\def\leq{\leqslant}
\def\b{\boldsymbol}
\def\N{\mathbb{N}}
\def\R{\mathbb{R}}
\def\C{\mathbb{C}}

\begin{document}
\title[Energy-critical quadratic Schr\"odinger system in $\mathbb{R}^6$]{On the energy-critical quadratic nonlinear Schr\"odinger system with three waves}

\author[F. Meng]{Fanfei Meng}
\address{Qiyuan Lab,  \ Beijing, \ China, \ 100095}
\email{mengfanfei@qiyuanlab.com}

\author[S. Wang]{Sheng Wang}
\address{Shanghai Center for Mathematical Sciences, Fudan University,  \ Shanghai, \ China, \ 200433}
\email{19110840011@fudan.edu.cn}
	
\author[C. Xu]{Chengbin Xu}
\address{School of Mathematics and Statistics, Qinghai Normal University,  \ Xining, Qinghai, \ China, \ 810008}
\email{xcbsph@163.com}

\date{\today}
\maketitle

\begin{abstract}
In this article, we consider the dynamics of the energy-critical quadratic nonlinear Schr\"odinger system
\[
\left\{
\begin{aligned}
& i u^1_t + \kappa_1 \Delta u^1 = -\overline{u^2}u^3, \\
& i u^2_t + \kappa_2 \Delta u^2 = -\overline{u^1}u^3, \\
& i u^3_t + \kappa_3 \Delta u^3 = -u^1u^2, \\
\end{aligned}
\right. \qquad (t, x) \in \R \times \R^6
\]
in energy-space $ {\dot H}^1 \times {\dot H}^1\times{\dot H}^1 $,  where the sign of potential energy can not be determined.
We prove the scattering theory with mass-resonance (or with radial initial data) below ground state via concentration compactness method.
We discover a family of new physically conserved quantities with mass-resonance which play an important role in the proof of scattering.
\end{abstract}

\maketitle
\section{Introduction}
Considering the Cauchy problem for the quadratic nonlinear Schr\"odinger system:
\begin{equation}\tag{$\text{NLS system}$}
\left\{
\begin{aligned}
& i \partial_{t}\b{\rm u} + A \b{\rm u}  =  \b{\rm f}(\b{\rm u}),\\
& \b{\rm u}(0,x)= \b{\rm u}_0(x),
\end{aligned}
\right. \qquad (t, x) \in \mathbb{R} \times \mathbb{R}^d,
\label{NLS system}
\end{equation}
where $ \b{\rm u} $ and $ \b{\rm u}_0 $ are all vector-valued complex functions with three components defined as follow:

\begin{equation}
\b{\rm u} := \left(
\begin{aligned}
u^1 \\
u^2 \\
u^3
\end{aligned}
\right), \quad \b{\rm u_0} := \left(
\begin{aligned}
u^1_0 \\
u^2_0 \\
u^3_0
\end{aligned}
\right),
\end{equation}
and  $ A $ is a $ 3 \times 3$ matrix, $ \b{\rm f} : \C^3 \to \C^3  $ as follow
\begin{equation}\label{mass-res}
A=
\begin{pmatrix}
 \kappa_1 \Delta & 0 & 0  \\
0                & \kappa_2 \Delta  & 0   \\
0                & 0                      & \kappa_3 \Delta
\end{pmatrix}, \quad \b{\rm f}(\b{\rm u}):= \left(
\begin{aligned}
-\overline{u^2}u^3 \\
-\overline{u^1}u^3 \\
-u^1 u^2
\end{aligned}
\right),
\end{equation}	
where $ u^{l} : \mathbb{R} \times \mathbb{R}^{d} \rightarrow \mathbb{C} $ are unknown functions and real numbers $ \kappa_{l} := \frac1{2 m_l} \in \left(0,\infty\right) $ are the inverses of three particle mass for $ l = 1, 2, 3 $.

\subsection{Background}
In terms of physics, the quadratic nonlinearity of \eqref{NLS system} arises from a model in the nonlinear optics describing the second harmonic generation and Raman amplification phenomenon in plasma.
This process is a nonlinear instability phenomenon (see \cite{Colin 2009} for more detail).
\eqref{NLS system} can be deduced from the following nonlinear Klein-Gordon system
\begin{equation}\tag{$\text{NLKG}$}
\left\{
\begin{aligned}
& \frac1{2c^2m_1} \partial_t^2 u^1 - \frac1{2m_1} \Delta u^1 + \frac{c^2m_1}2 u^1 = - \mu_1 \overline{u^2} u^3, \\
& \frac1{2c^2m_2} \partial_t^2 u^2 - \frac1{2m_2} \Delta u^2 + \frac{c^2m_2}2 u^2 = - \mu_2 \overline{u^1} u^3, \\
& \frac1{2c^2m_3} \partial_t^2 u^3 - \frac1{2m_3} \Delta u^3 + \frac{c^2m_3}2 u^3 = - \mu_3 u^1 u^2,
\end{aligned}
\right. \qquad (t, x) \in \R \times \R^d
\label{NLKG}
\end{equation}
where $ c = 3.0 \times 10^8 m/s $ is the velocity of light, $ u^l : \R \times \R^d \to \C $ is the wave function, $ m_l > 0 $ is the mass of particle $ u^l $, and the three $ \mu_l $ have the same sign for $ l = 1, 2, 3 $.
Define
\begin{equation}
\left(
\begin{aligned}
u^1_c \\
u^2_c \\
u^3_c
\end{aligned}
\right) := \left(
\begin{aligned}
e^{itc^2m_1} u^1 \\
e^{itc^2m_2} u^2 \\
e^{itc^2m_3} u^3
\end{aligned}
\right),
\end{equation}
and we can compute the derivatives in \eqref{NLKG} using $ u^l = e^{-itc^2m_l} u^l_c $ to get
\begin{equation*}
\begin{aligned}
\partial_t u^l = & ~ -ic^2m_l e^{-itc^2m_l} u^l_c + e^{-itc^2m_l} \partial_t u^l_c, \\
\partial_t^2 u^l = & ~ -c^4m_l^2 e^{-itc^2m_l} u^l_c - 2ic^2m_l e^{-itc^2m_l} \partial_t u^l_c + e^{-itc^2m_l} \partial_t^2 u^l_c.
\end{aligned}
\end{equation*}
Then,
\begin{equation*}
\frac1{2c^2m_l} \partial_t^2 u^l + \frac{c^2m_l}2 u^l = e^{-itc^2m_l} (-i \partial_t u^l_c + \frac1{2c^2m_l} \partial_t^2 u^l_c), \qquad \forall ~ l = 1, 2, 3.
\end{equation*}
and so
\begin{equation}\tag{$\text{NLW/NLS}$}
\left\{
\begin{aligned}
& \frac1{2c^2m_1} \partial_t^2 u^1_c - i \partial_t u^1_c - \frac1{2m_1} \Delta u^1_c = - e^{itc^2(m_1 + m_2 - m_3)} \mu_1 \overline{u^2_c} u^3_c, \\
& \frac1{2c^2m_2} \partial_t^2 u^2_c - i \partial_t u^2_c - \frac1{2m_2} \Delta u^2_c = - e^{itc^2(m_2 + m_1 - m_3)} \mu_2 \overline{u^1_c} u^3_c, \\
& \frac1{2c^2m_3} \partial_t^2 u^3_c - i \partial_t u^3_c - \frac1{2m_3} \Delta u^3_c = - e^{itc^2(m_3 - m_1 - m_2)} \mu_3 u^1_c u^2_c.
\end{aligned}
\right.
\label{NLW/NLS}
\end{equation}
Under the mass resonance condition
\begin{equation}\tag{$\text{mass resonance}$}
m_1 + m_2 = m_3,
\label{mass resonance}
\end{equation}
we take non-relativistic limit, i.e., $ c \to \infty $, to obtain
\begin{equation}\tag{$\text{NLSm}$}
\left\{
\begin{aligned}
& i \partial_t u^1_c + \frac1{2m_1} \Delta u^1_c = \mu_1 \overline{u^2_c} u^3_c, \\
& i \partial_t u^2_c + \frac1{2m_2} \Delta u^2_c = \mu_2 \overline{u^1_c} u^3_c, \\
& i \partial_t u^3_c + \frac1{2m_3} \Delta u^3_c = \mu_3 u^1_c u^2_c.
\end{aligned}
\right.
\label{NLSm}
\end{equation}
Lastly, constructing new functions
\begin{equation}
\left(
\begin{aligned}
v^1 (t, x) \\
v^2 (t, x) \\
v^3 (t, x)
\end{aligned}
\right) := \left(
\begin{aligned}
- \sqrt{\mu_2 \mu_3} u^1_c (t, \frac{x}{\sqrt{2m_1}}) \\
- \sqrt{\mu_1 \mu_3} u^2_c (t, \frac{x}{\sqrt{2m_2}}) \\
- \sqrt{\mu_1 \mu_2} u^3_c (t, \frac{x}{\sqrt{2m_3}})
\end{aligned}
\right),
\label{v}
\end{equation}
and denoting
\begin{equation}
\kappa_l = \frac1{2m_l}, \qquad l = 1, 2, 3,
\label{kappa}
\end{equation}
we have \eqref{NLS system} finally.

\begin{remark}
\eqref{NLS system} is the non-relativistic limit of quadratic nonlinear Klein-Gordon system both focusing and defocusing.
Indeed, we can find from \eqref{v} that the three $ \mu_l, l = 1, 2, 3 $ have the same sign in \eqref{NLKG}.
\end{remark}

It is important that \eqref{NLS system} is the form of infinite dimensional Hamilton if we treat $ \C \sim \R^2 $.
\begin{equation}\tag{$\text{IDH}$}
\partial_{t} \b{\rm u} = S_6 H^{\prime}(\b{\rm u}),
\label{IDH}
\end{equation}
where $ S_6 $ is the standard $ 6 \times 6 $ symplectic matrix and $ H = H(\b{\rm u}) $ is the Hamiltonian with
\begin{equation}
H(\b{\rm u}) = \sum_{i=1}^3 \frac{\kappa_i}{2}\Vert u^i \Vert_{{\dot H}^1}^2 - \Re \int_{\mathbb{R}^d} \overline{u^1}\overline{u^2}u^3 {\rm d}x.
\label{Hamiltonian}
\end{equation}
Thus, by Emmy Noether's theorem, the solutions to \eqref{NLS system} conserve the \emph{mass}, \emph{energy} and \emph{momentum}, corresponding to the transfomations of rotation of phase, translation of time and translation of space respectively, defined as follow:
\begin{equation*}
\begin{aligned}
M(\b{\rm u}) :& = \frac{1}{2}\Vert u^1 \Vert_{L^2}^2 + \frac{1}{2} \Vert u^2 \Vert_{L^2}^2 +\Vert u^3 \Vert_{L^2}^2  \equiv M(\b{\rm u}_{0}), \\
E(\b{\rm u}) :& = K(\b{\rm u}) -V(\b{\rm u}) \equiv E(\b{\rm u}_{0}), \\
P(\b{\rm u}) :& = \Im \int_{\mathbb{R}^d} \left( \overline{u^1}\nabla u^1 + \overline{u^2}\nabla u^2 + \overline{u^3}\nabla u^3 \right) {\rm d}x \equiv P(\b{\rm u}_0),
\end{aligned}
\end{equation*}
where
\[
\begin{aligned}
& (\text{kinetic energy}) \quad & K(\b{\rm u}): & = \frac{\kappa_1}{2}\Vert u^1 \Vert_{{\dot H}^1}^2 + \frac{\kappa_2}{2} \Vert u^2 \Vert_{{\dot H}^1}^2 + \frac{\kappa_3}{2} \Vert u^3 \Vert_{{\dot H}^1}^2  , \\
& (\text{potential energy}) \quad & V(\b{\rm u}) :& = \Re \int_{\mathbb{R}^d} \overline{u^1}\overline{u^2}u^3 {\rm d}x.
\end{aligned}
\]

\begin{remark}
On one hand, we can compute $ \partial_t \b{\rm u} $ by \eqref{NLS system}
\begin{equation*}
\partial_t \b{\rm u} := \left(
\begin{aligned}
\partial_t u^1 \\
\partial_t u^2 \\
\partial_t u^3
\end{aligned}
\right) \sim \left(
\begin{aligned}
\partial_t u^{1, 1} \\
\partial_t u^{1, 2} \\
\partial_t u^{2, 1} \\
\partial_t u^{2, 2} \\
\partial_t u^{3, 1} \\
\partial_t u^{3, 2}
\end{aligned}
\right) = \left(
\begin{aligned}
-\kappa_1\Delta u^{1, 2} - u^{2, 1}u^{3, 2} + u^{2, 2}u^{3, 1} \\
\kappa_1\Delta u^{1, 1} + u^{2, 1}u^{3, 1} + u^{2, 2}u^{3, 2} \\
-\kappa_2\Delta u^{2, 2} - u^{1, 1}u^{3, 2} + u^{1, 2}u^{3, 1} \\
\kappa_2\Delta u^{2, 1} + u^{1, 1}u^{3, 1} + u^{1, 2}u^{3, 2} \\
-\kappa_3\Delta u^{3, 2} - u^{1, 1}u^{2, 2} - u^{1, 2}u^{2, 1} \\
\kappa_3\Delta u^{3, 1} + u^{1, 1}u^{2, 1} - u^{1, 2}u^{2, 2}
\end{aligned}
\right),
\end{equation*}
where $ u^{l, 1} := \Re u^l $ and $ u^{l, 2} := \Im u^l $ for $ l = 1, 2, 3 $.

On the other hand, we can calculate the variation by \eqref{Hamiltonian} and
\begin{equation*}
\langle H^{\prime}(\b{\rm u}), \b{\rm v} \rangle = \left. \frac{\rm d}{{\rm d}\varepsilon} \right\vert_{\varepsilon=0} H(u^1+\varepsilon v^1, u^2+\varepsilon v^2, u^3+\varepsilon v^3)
\end{equation*}
where $ \b{\rm v} \in C_c^{\infty}(\R^6) $ to know
\begin{equation*}
H^{\prime}(\b{\rm u}) = \left(
\begin{aligned}
-\kappa_1\Delta u^{1, 1} - u^{2, 1}u^{3, 1} - u^{2, 2}u^{3, 2} \\
-\kappa_1\Delta u^{1, 2} + u^{2, 2}u^{3, 1} - u^{2, 1}u^{3, 2} \\
-\kappa_2\Delta u^{2, 1} - u^{1, 1}u^{3, 1} - u^{1, 2}u^{3, 2} \\
-\kappa_2\Delta u^{2, 2} + u^{1, 2}u^{3, 1} - u^{1, 1}u^{3, 2} \\
-\kappa_3\Delta u^{3, 1} - u^{1, 1}u^{2, 1} - u^{1, 2}u^{2, 2} \\
-\kappa_3\Delta u^{3, 2} + u^{1, 1}u^{2, 2} - u^{1, 2}u^{2, 1}
\end{aligned}
\right).
\end{equation*}

To sum up, we know $ S_6 $ in \eqref{IDH} is
\begin{equation*}
S_6 := diag(S_2; S_2; S_2) =
\begin{pmatrix}
0 & -1 & 0 & 0 & 0 & 0 \\
1 & 0 & 0 & 0 & 0 & 0 \\
0 & 0 & 0 & -1 & 0 & 0 \\
0 & 0 & 1 & 0 & 0 & 0 \\
0 & 0 & 0 & 0 & 0 & -1 \\
0 & 0 & 0 & 0 & 1 & 0 \\
\end{pmatrix}.
\end{equation*}
\end{remark}

Furthermore, this system is invariant under the gauge transform $ (u^1, u^2, u^3) \to (e^{i\theta_1} u^1, e^{i\theta_2} u^2, e^{i(\theta_1 + \theta_2)} u^3) $ and the following scaling
$$\b{\rm u}_{\lambda}(t,x):= \lambda^2 \b{\rm u} (\lambda^2 t, \lambda x), \quad \lambda > 0$$
which means that $\b{\rm u}_{\lambda}(t,x)$ is also a solution of \eqref{NLS system}. Computing the homogeneous Sobolev norm,
$$\|\b{\rm u}_{\lambda}(\cdot,0)\|_{\dot{H^s}(\mathbb{R}^d)}=\lambda^{s-\frac{d}{2}+2}\|\b{\rm u}_0\|_{\dot{H^s}(\mathbb{R}^d)},$$
This gives the critical index
$$s_c= \frac{d}{2} - 2$$
which leaves the scaling symmetry invariant.
The case $s_c=1$ is called energy-critical and $0<s_c <1$ is called mass-energy inter-critical.

\subsection{Methods and results}
Now, we review some basic facts about the nonlinear Schr\"odinger equation:
\[\tag{$\text{NLS}$}
i\partial_{t}u + \Delta u = \mu \vert u \vert^{p-1} u , \quad (t, x) \in \mathbb{R} \times \mathbb{R}^d,
\label{NLS}
\]
The case $\mu=1$ is the defocusing case, while $\mu=-1$ gives the focusing case.
For the defoucsing case, the potential energy is positive, while for the focusing case, the potential energy is negative.
In contrast to the classical nonlinear  Schr\"odinger equation \eqref{NLS}, the \eqref{NLS system} do not have classification of the so-called focusing and defocusing cases, since we can not figure out the sign  of potential energy.

In this article, we study the energy-critical \eqref{NLS system} in energy space, i.e, $d=6$. For convenience, we introduce the following notation:
\begin{align}\nonumber
	{{\dot{\rm H}}^s}:= {\dot H}^s \times {\dot H}^s \times {\dot H}^s \quad s\ge 0, \\ \nonumber
	{\rm L}^p:=  L^p \times L^p\times L^p \quad p\ge 1.
\end{align}
We also denote the Schr\"odinger flow $ \verb"S"(t) $ as follows
$$
\verb"S"(t) :=
\begin{pmatrix}
e^{ \kappa_1 it\Delta} & 0 & 0  \\
0                & e^{\kappa_2 it\Delta}  & 0   \\
0                & 0                      & e^{\kappa_2 it\Delta}
\end{pmatrix}.
$$

\begin{definition}(Solution)
A function $\b{\rm u}: I\times \mathbb{R}^6 \to \C^3$  on a nonempty time interval $ 0\in I \subset \R$ is a solution to \eqref{NLS system} if it lies in the class $C^{0}_{t}(I, {\rm\dot H}_{x}^1(\mathbb{R}^6)) \cap {\rm L}_{t, x}^4(I \times \mathbb{R}^6)$, and satisfies Duhamel formula
\[
\b{\rm u}(t) = \verb"S"(t) \b{\rm u}_0 + i \int_{0}^{t} \verb"S"(t - \tau) \b{\rm f}(\b{\rm u}(\tau)) {\rm d}\tau .
\label{def}
\]
We refer to the interval $I$ as the lifespan of $u$. We say that $u$ is a maximal-lifespan solution if solution cannot be extended to any strictly larger interval. We say that $u$ is a global solution if $I=\R$.
\end{definition}
\begin{definition}[Scattering size, \cite{Killip2010}]\label{scattering size}
	The \emph{scattering size} of a solution $ \b{\rm u} = (u^1, u^2,u^3)^T $ to \eqref{NLS system} on the interval $ 0 \in I $ is
	\[
	S_I (\b{\rm u}) = \Vert \b{\rm u} \Vert_{{\rm L}_{t, x}^4(I \times \R^6)}^4 \sim \int_I \int_{\R^6} \vert u^1(t, x) \vert^4 + \vert u^2(t, x) \vert^4 +\vert u^3(t, x) \vert^4 {\rm d}x {\rm d}t.
	\]
\end{definition}
\begin{definition}(Blow-up)
	We say that a solution $\b {\rm u}$ to \ref{NLS system} blows up forward in time if there exists a time $t_1 \in I$ such that
	\[
	S_{[t_1,\sup I)} (\b{\rm u}) = \infty,
	\]
	and that $\b {\rm u}$ blows up backward in time if there exists a time $t_1 \in I$ such that
	\[
	S_{(\inf I,t_1]} (\b{\rm u}) = \infty.
	\]
\end{definition}
\begin{definition}(Scattering)
	We say a global solution $\b{\rm u}$ is scattering if there exists $\b{\rm u}_{\pm} \in {\rm\dot H}^1(\mathbb{R}^6)$ such that
	\begin{equation*}
	\lim_{t \to \pm \infty} \left\Vert \b{\rm u}(t) - \verb"S"(t) \b{\rm u}_{\pm} \right\Vert_{{\rm\dot H}^1(\mathbb{R}^6)} = 0.
	\end{equation*}
\end{definition}
 \begin{remark}[$ {\rm\dot H}^1 $ scattering]\label{sc}
	It is easy to check that the solution to \eqref{NLS system} will scatter in $ {\rm\dot H}^1(\mathbb{R}^6) $ as $ t \to + \infty $, if the solution $ \b{\rm u}(t) $ to \eqref{NLS system} is global in $ {\rm\dot H}^1(\mathbb{R}^6) $ with finite scattering size on $ \R $, i.e., $ S_{\R} (\b{\rm u}) < + \infty $.
\end{remark}
Similarly to the classical \eqref{NLS}, we can define the Galilean transformation to the \eqref{NLS system} as follows:
\[
\left(
\begin{aligned}
u^1(t, x) \\
u^2(t, x) \\
u^3(t,x)
\end{aligned}
\right) \to \left(
\begin{aligned}
e^{i\frac{x\cdot\xi}{\kappa_1}}e^{-it\frac{\vert\xi\vert^2}{\kappa_1}}u^1(t, x-2t\xi) \\
e^{i\frac{x\cdot\xi}{\kappa_2}}e^{-it\frac{\vert\xi\vert^2}{\kappa_2}}u^2(t, x-2t\xi) \\
e^{i\frac{x\cdot\xi}{\kappa_3}}e^{-it\frac{\vert\xi\vert^2}{\kappa_3}}u^3(t, x-2t\xi)
\end{aligned}
\right).
\]
It is easy to check that the \eqref{NLS system} is invariance under Galiean transformation only if the coefficient satisfies
\begin{equation}
\frac{1}{\kappa_3}=\frac{1}{\kappa_1} + \frac{1}{\kappa_2}.
\label{mr}
\end{equation}
By \eqref{mass resonance} and \eqref{kappa}, we know \eqref{mr} is the same condition as mass resonance.
Besides, there are some reference about Galilean transformation and mass-resonance, such as \cite{Ardila2021}, \cite{Gao2021}, \cite{Meng2021}.

\begin{remark}
Moreover, we can explore a family of new physically conserved quantities under mass-resonance:
\begin{equation} \label{New}
\Lambda(\b{\rm u}) := \frac{1}{2\kappa_1} \Vert u^1 \Vert^2_{L^2} + \frac{1}{2\kappa_2} \Vert u^2 \Vert^2_{L^2} + \frac{1}{2\kappa_3} \Vert u^3 \Vert^2_{L^2}
\end{equation}
by clever observation, which plays a key role to exclude the  soliton-type solution.
More precisely, in subsection \ref{SS}, we rely on the new physically conserved quantity to prove the momentum equals to zero.
Then we obtain some compactness conditions.
Combining these facts, we get a result which is contradict with the characteristics of solitons.

It is obvious that the conserved mass $ M(\b{\rm u}) $ is a special $ \Lambda(\b{\rm u}) $ when $ (\kappa_1, \kappa_2, \kappa_3) = (1, 1, \frac12) $.
Actually, we are even able to define infinite conserved quantities which can be regarded as the generalized of \eqref{New}:
\[
\Lambda_{a, b}(\b{\rm u}) := \frac{a}{2} \Vert u^1 \Vert^2_{L^2} + \frac{b}{2} \Vert u^2 \Vert^2_{L^2} + \frac{a+b}{2} \Vert u^3 \Vert^2_{L^2}, \qquad \forall ~ a, b \in \R.
\]
\end{remark}

Next, we recall some research advances about \eqref{NLS} and \eqref{NLS system}.
In terms of \eqref{NLS}, Bourgain \cite{Bourgain1999} first solved the defocusing energy-critical with radial initial date by energy induction technique.
Based on energy induction technique,  Colliander--Keel--Staffilani--Takaoka--Tao \cite{Colliander2006} introduced the interaction Morawetz estimate to get the well-posedness and scattering for any initial date.
Later, Visan \cite{Visan2007} established the similar results for higher dimensions.
The energy-critical focusing problem has also received a lot of attention.
Kenig--Merle \cite{Kenig2006} first solved this problem with radial case using concentration compactness/rigidity arguments.
Inspired by Kenig--Merle, Killip--Visan \cite{Killip2010} proved the scattering theory for any initial date  in dimension  $d \ge 5$. In addition,  another interesting piece of work was that Miao--Murphy--Zheng \cite{Miao2014} studied the defocusing energy-supcritical in four dimensions.

For the study of \eqref{NLS system}, there are a lot of noteworthy work. Ogawa--Uriya\cite{Ogawa2015}, Ardila \cite{Ardila2018} and Pastor \cite{Pastor2019} studied the 2-dimensional and 1-dimensional \eqref{NLS system} with three wave interaction respectively.
Gao--Meng--Xu--Zheng \cite{Gao2021} proved the scattering theory for the \eqref{NLS system} with two wave interaction in dimension six. For the energy-subucritical case, such as \cite{Hamano2018, Hamano2019, Meng2021, Noguera2020},
they gave a series of results about the well-posedness and dynamics of the \eqref{NLS system}.
Especially, an attractive work was that Ardila--Dinh--Forcella \cite{Ardila2021} studied the energy-subcritical \eqref{NLS system} with two wave interaction for cubic-type nonlinearity under mass-resonance conditions.
In addition, Dinh--Forcella \cite{Dinh 2021} studied the blow-up mechanism about the \eqref{NLS system} with two wave interaction.

In this work, we focus on the energy-critical \eqref{NLS system} with three wave interaction. We will present a classification about the dynamic behavior of the solutions, the fact that they scatter and blow up depends on the relationship between the initial value and the ground state $ \b{\rm W} := (\phi_1, \phi_2,\phi_3)^T \in \Xi$, where $(\phi_1, \phi_2,\phi_3)^T$ is the non-nogentive radial  solution for the following elliptic system:

\begin{equation}
\left\{
\begin{aligned}
& - \kappa_1\Delta \phi_1 = \phi_2 \phi_3, \\
& - \kappa_2\Delta \phi_2 = \phi_1 \phi_3, \\
& - \kappa_3\Delta \phi_3 = \phi_1 \phi_2
\end{aligned}
\right.
\quad x \in \mathbb{R}^6,
\label{gs}
\end{equation}
where $ \Xi :=\left\{ \b{\rm G} \in {\rm\dot H}^1(\R^6)  ~ \big\vert ~ V(\b{\rm G}) \neq 0 \right\} $. And we can show that the ground state is $\b{\rm W} := (\frac{\phi_0}{\sqrt{\kappa_1}}, \frac{\phi_0}{\sqrt{\kappa_2}},\frac{\phi_0}{\sqrt{\kappa_3}})^T$, where $\phi_0$ is a ground state of equation:
\begin{equation}
- \sqrt{ \kappa_1\kappa_2\kappa_3}\Delta \phi_0 = \phi_0^2.
\label{6}
\end{equation}

Then, we introduce our main results:

\begin{theorem}[Bounds of scattering size] \label{Sb}
Let $\b{\rm u}$ : $I\times \R^6 \to \C^3$ be a solution to \eqref{NLS system} satisfying the mass-resonance (or $\b{\rm u}$ is radial without mass-resonance).  If
\begin{equation}
E(\b{\rm u}_0) < E (\b{\rm W}) \quad \text{and} \quad K(\b{\rm u}_0) < K(\b{\rm W}),
\label{EK}
\end{equation}
then
\begin{align*}
\int_I \int_{\R^6} \vert \b{\rm u}(t, x) \vert^4 {\rm d}x {\rm d}t  < \infty.
\end{align*}
\end{theorem}

\begin{corollary}[Scattering]\label{CSb}
Let $ \b{\rm u} $ be a solution to \eqref{NLS system} with maximal lifespan $ I_{\max} $  satisfying the mass-resonance (or $\b{\rm u}$ is radial without mass-resonance). Assume also that \eqref{EK} holds.
Then $ I_{\max} = \R $ and the solution $ \b{\rm u} $ scatters with $ S_{\R}(\b{\rm u}) < \infty $.
\end{corollary}

It is easy to establish the scattering theory for the \eqref{NLS system} if the Corollary \ref{CSb} is true.
At the same time, the energy-critical \eqref{NLS system} is global well-posedness in $ {\rm\dot H}^1(\R^6) $.

\subsection{Outline of the article}
This article is organized as follows:
In Section \ref{PRE}, we introduce some preliminaries, such as some basic harmonic analysis, Strichartz estimates, local well-posedness, linear profile decomposition, Virial-type identity.

In particular, the linear profile decomposition for Schr\"odinger flow, which relies on the inverse Strichartz inequality in Proposition \ref{ISI}, is built in Theorem \ref{lpd}.
It is a powerful tool for recovering compactness, and is often used in conjunction with the following nonlinear profile.

\begin{definition}
Let $ \b{\rm v}(t, x) = \verb"S"(t) \b{\rm v}_0 $ be the solution to linear Schr\"odinger system with initial data $ \b{\rm v}_0 \in {\rm\dot H}^1(\R^6) $, let $ \{ t_n \}_{n=1}^{\infty} $ be a time series satisfying
\[
t_n \to t_{\infty} \in \R \cup \{ \pm \infty \}, \qquad n \to \infty.
\]
We call $ \b{\rm u} $ is the \emph{nonlinear profile} associated to $ (\b{\rm v}_0, \{ t_n \}_{n=1}^{\infty}) $, if there exists a time interval
\[
I \ni t_{\infty}, \qquad (\text{if} ~ \vert t_{\infty} \vert = \infty, ~ \text{then} ~ I = [a, +\infty) ~ \text{or} ~ I = (-\infty, a])
\]
such that $ \b{\rm u} $ is a solution to \eqref{NLS system} defined on $ I $ and
\[
\lim_{n \to \infty} \Vert \b{\rm u}(t_n, \cdot) - \b{\rm v}(t_n, \cdot) \Vert_{{\rm\dot H}^1(\R^6)} = 0.
\]
\end{definition}

After the establishment of small data theory (Theorem \ref{sd}), one can ensure that the nonlinear profile above is well-defined, i.e., one can prove the existence and uniqueness of it, see Corollary \ref{ENP} and Corollary \ref{UNP}.

In Section \ref{VA}, we study the analysis of characterization to the ground state in \eqref{gs}, which will be used to establish the energy trapping theory.

It is worth mentioning that our proof is different from Kenig's strategy in \cite{Kenig2006}, which is caused by the specificity of \eqref{NLS system} -- the sign of the potential energy $ V(\b{\rm u}) $ cannot be determined.
More specifically, on one hand, in establishing $ E(\b{\rm u}) \ge K(\b{\rm u}) $, it is no longer obvious due to the absence of $ V(\b{\rm u}) \ge 0 $, but fortunately the establishment of the coercivity of energy, Proposition \ref{coer}, can overcome this difficulty.
On the other hand, in establishing $ E(\b{\rm u}) \le K(\b{\rm u}) $, we need to control the absolute value of the potential energy $ \vert V(\b{\rm u}) \vert $.
While the Gagliardo-Nirenberg inequality, which degenerates from an interpolation type inequality to an embedded type inequality, only controls the potential energy $ V(\b{\rm u}) $, see Proposition \ref{GN}.
We observe that if we substitute soliton solutions $ \b{\rm u} = (u^1, u^2, u^3)^T = (e^{iat}\phi_1, e^{ibt}\phi_2, e^{i(a+b)t}\phi_3)^T $, a particular family of solutions generated by the ground state $ \b{\rm W}=(\phi_{1}, \phi_{2}, \phi_{3})^{T} $, into the potential energy, then
\[
V(\b{\rm u}) = \Re \int_{\R^6} \overline{u^1} \overline{u^2} u^3 {\rm d}x = \int_{\R^6} \phi_1 \phi_2 \phi_3 {\rm d}x.
\]
The special complex coupling mechanism $ \Re(\overline{u^1} \overline{u^2} u^3) $ in the above reminds us that the control of $ \vert V(\b{\rm u}) \vert $ can be completed by constructing the Arithmetic-Ceometric Mean-Value inequality, see the proof of Lemma \ref{Ck} for more details.

In Section \ref{GWPS}, we complete the proof of Theorem \ref{Sb} via concentration compactness method and rigidity argument.
An important concept is almost periodic solution written in Definition \ref{almost periodic}, whose existence is the key point of contradiction in the proof of scattering.
After showing the existence of critical solution which is also almost periodicity modulo symmetries in Subsection \ref{eocs}, we then ruled out all three possibilities of it in the later subsections.
In particular, the proof of the negative regularity of global almost periodic solution $ \b{\rm u}_c $ takes up an entire part in Subsection \ref{nr}, which is very powerful in the exclusion of soliton-type and low-to-high frequency cascade solutions of \eqref{NLS system}.

For the sake of completeness of this paper, we show another derivation of model \eqref{NLS system} from the view of nonlinear optics in Appendix \ref{AD}.
We collect and show the blow-up result about the \eqref{NLS system} in Appendix \ref{BU}.

\section{Preliminaries}\label{PRE}
We conclude the introduction by giving some notations which
will be used throughout this paper. We always use $X\lesssim Y$ to denote $X\leq CY$ for some constant $C>0$.
Similarly, $X\lesssim_{u} Y$ indicates that there exists a constant $C:=C(u)$ depending on $u$ such that $X\leq C(u)Y$.
We also use the notation $X\sim Y$ to denote $X \lesssim Y \lesssim X$.
%Let $P_{\le N},\ P_{\ge N},\ P_{N}$ and $ P_{M<\cdot\le N}$ denote the standard Littlewood--Paley operators which commute with the propagator $\verb"S"(t)$.

\subsection{Basic harmonic analysis}
Let $ \psi(\xi) $ be a radial smooth function supported in the ball $ \{ \xi \in \R^6 : \vert \xi \vert \le \frac{11}{10} \} $ and equal to 1 on the ball $ \{ \xi \in \R^6 : \vert \xi \vert \le 1 \} $.
For each number $ N > 0 $, we define the Fourier multipliers

\begin{equation*}
\begin{aligned}
\widehat{P_{\le N} \b{\rm g}} (\xi) & ~ := \psi \left( \frac{\xi}{N} \right) \hat{\b{\rm g}}(\xi), \\
\widehat{P_{> N} \b{\rm g}} (\xi) & ~ := \left( 1 - \psi \left( \frac{\xi}{N} \right) \right) \hat{\b{\rm g}}(\xi), \\
\widehat{P_N \b{\rm g}} (\xi) & ~ := \left( \psi \left( \frac{\xi}{N} \right) - \psi \left( \frac{2\xi}{N} \right) \right) \hat{\b{\rm g}}(\xi),
\end{aligned}
\end{equation*}
and similarly $ P_{< N} $ and $ P_{\ge N} $. We also define
\[
P_{M < \cdot \le N} := P_{\le N} - P_{\le M} = \sum_{M < N^\prime \le N} P_{N^\prime}
\]
whenever $ M < N $.
We usually use this multipliers when $M$ and $N$ are dyadic numbers. It is worth saying that all summations over $M$ or $N$ are understood to be over dyadic numbers. Like all Fourier multipliers, the Littlewood--Paley operators commute with the Sch\"odinger flow $\verb"S"(t)$.

As some applications of the Littlewood--Paley theory, we have the following lemma.

\begin{lemma}[Bernstein inequalities]\label{Bern}
For any $ 1 \le p \le q \le \infty $ and $ s \ge 0 $,
\begin{equation*}
\begin{aligned}
\Vert P_{\ge N} \b{\rm g} \Vert_{{\rm L}_x^p(\R^6)} \lesssim & ~ N^{-s} \Vert \vert \nabla \vert^s P_{\ge N} \b{\rm g} \Vert_{{\rm L}_x^p(\R^6)} \lesssim N^{-s} \Vert \Vert \nabla \vert^s \b{\rm g} \Vert_{{\rm L}_x^p(\R^6)}, \\
\Vert \vert \nabla \vert^s P_{\le N} \b{\rm g} \Vert_{{\rm L}_x^p(\R^6)} \lesssim & ~ N^s \Vert P_{\le N} \b{\rm g} \Vert_{{\rm L}_x^p(\R^6)} \lesssim N^s \Vert \b{\rm g} \Vert_{{\rm L}_x^p(\R^6)}, \\
\Vert \vert \nabla \vert^{\pm s} P_N \b{\rm g} \Vert_{{\rm L}_x^p(\R^6)} \lesssim & ~ N^{\pm s} \Vert P_N \b{\rm g} \Vert_{{\rm L}_x^p(\R^6)} \lesssim N^{\pm s} \Vert \b{\rm g} \Vert_{{\rm L}_x^p(\R^6)} ,\\
\Vert P_N \b{\rm g} \Vert_{{\rm L}_x^q(\R^6)} \lesssim & ~ N^{\frac{6}{p} - \frac{6}{q}} \Vert P_N \b{\rm g} \Vert_{{\rm L}_x^p(\R^6)},
\end{aligned}
\end{equation*}
where $ \vert \nabla \vert^s $ is the classical fractional-order operator.
\end{lemma}

\begin{lemma}[Fractional chain rule, \cite{Christ1991}] \label{cr}
	Suppose $G \in C^1\left(\C\right)$, $s\in\left(0,1\right]$, and $1<p, p_1, p_2< \infty$ are such that $\frac{1}{p}= \frac{1}{p_1} + \frac{1}{p_2}$. Then,
	\[\
	\left\Vert \vert\nabla\vert^s G(u) \right\Vert_{L_x^p}  \lesssim \left \Vert G^{\prime}(u) \right \Vert_{L_x^{p_1}} \left \Vert \vert \nabla \vert ^s u\right\Vert _{L_x^{p_2}}
	\]
\end{lemma}

\subsection{Local well-posedness}
Let $ \Lambda_s $ be the set of pairs $ (p, q) $ with $ q \geqslant 2 $ and satisfying
\begin{equation}
\frac{2}{q} = 6 \left( \frac{1}{2} - \frac{1}{r} \right) - s.
\label{Lambda}
\end{equation}
Define
\begin{equation*}
\Vert \b{\rm u} \Vert_{\mathcal{S}({\rm\dot H}^s(\mathbb{R}^6))}:= \sup_{(q, r) \in \Lambda_{s}} \Vert \b{\rm u} \Vert_{{\rm L}_t^q(\mathbb{R}, {\rm L}_x^r(\mathbb{R}^6))},
\end{equation*}
and
\begin{equation*}
\Vert \b{\rm u} \Vert_{\mathcal{S}^\prime({\rm\dot H}^s(\mathbb{R}^6))}:= \inf_{(q, r) \in \Lambda_{s}} \Vert \b{\rm u} \Vert_{{\rm L}_t^{q^\prime}(\mathbb{R}, {\rm L}_x^{r^\prime}(\mathbb{R}^6))}.
\end{equation*}
We extent our notation $ \mathcal{S}({\rm\dot H}^s(\mathbb{R}^6)), \mathcal{S}^\prime({\rm\dot H}^s(\mathbb{R}^6)) $ as follows: we write $ \mathcal{S}(I, {\rm\dot H}^s(\mathbb{R}^6)) $ or $ \mathcal{S}^\prime(I, {\rm\dot H}^s(\mathbb{R}^6)) $ to indicate a restriction to a time subinterval $ I \subset \mathbb{R} $.

\begin{lemma}[Dispersive estimate]\label{disper}
	For any $ t \ne 0 $ and $ r \ge 2 $, we have
	\begin{equation*}
	\begin{aligned}
	\Vert \verb"S"(t) \b{\rm g} \Vert_{{\rm L}_x^\infty(\R^6)} \lesssim_ {\kappa_1 \kappa_2 \kappa_3} & ~  \vert t \vert^{-3} \Vert \b{\rm g} \Vert_{{\rm L_x}^1(\R^6)},
\end{aligned}
\end{equation*}
	where $ r^\prime $ is the H\"older conjugation index, that is, $ \frac{1}{r} + \frac{1}{r^\prime} = 1 $.
\end{lemma}

We can use $ TT^\star $ argument and combine the endpoint case \cite{Keel1998} to obtain the following Strichartz estimates.
\begin{theorem}[Strichartz estimates, \cite{Meng2021}]\label{strichartz}
	The solution $ \b{\rm u} $ to \eqref{NLS system} on an interval $ I \ni t_{0} $ obeys
	\begin{equation}
	\Vert \b{\rm u} \Vert_{\mathcal{S}(I, {\rm L}^2(\mathbb{R}^6))} \lesssim \left( \Vert \b{\rm u}(t_{0}) \Vert_{{\rm L}^2(\mathbb{R}^6)} + \Vert \b{\rm f}(\b{\rm u}) \Vert_{\mathcal{S}^\prime(I, {\rm L}^2(\mathbb{R}^6))} \right).
	\label{Strichartz}
	\end{equation}
\end{theorem}

Similar to the classical case, we can also establish the local well-posedness for \eqref{NLS system} by fixed point method. For the energy-critical circumstance, it's worth noting that the lifespan of the solution depends not only on the size of the ${{\dot{\rm H}}^1(\R^6)}$ norm, but also on the profiles of the initial date.

\begin{theorem}(Local well-posedness, \cite{Cazenave2003}) \label{sd}
Fixed $ \b{\rm u}_0 \in {\rm\dot H}^1 (\mathbb{R}^6) $, there exists a unique maximal-lifespan solution $ \b{\rm u} : I \times \R^6 \to \C^3 $ to \eqref{NLS system} with initial data $\b{\rm u}(0)=\b{\rm u}_0$ satisfying the following properties:

\begin{enumerate}
	\item $ 0 \in I $ is an open interval.
	\item If $ \sup I $ is finite, then $ \b{\rm u} $ blows-up forward in time. Similarly, if $ \inf I $ is finite, then $ \b{\rm u} $ blows-up backward in time.
	\item If $ \sup I = +\infty $ and $ \b{\rm u} $ does not blow-up forward in time, then $ \b{\rm u} $ scatters forward in time.
	Conversely, given $ \b{\rm u}_+ \in {\rm\dot H}^1(\R^6) $, there is a unique solution $ \b{\rm u}(t) $ to \eqref{NLS system} in a neighborhood of $ t = \infty $ such that
	\begin{equation*}
	\lim_{t \to +\infty} \left\Vert \b{\rm u}(t) - \verb"S"(t) \b{\rm u}_+ \right\Vert_{{\rm\dot H}^1(\mathbb{R}^6)} = 0.
	\end{equation*}
	Analogous statements hold backward in time.
	\item There exists a small number $ \delta_{sd} > 0 $ satisfying that
	if $ \| \b{\rm u}_0 \|_{\mathcal{S}(\mathbb{R}, {\rm\dot H}^1(\mathbb{R}^6))} \le \delta_{sd} $, then $ \b{\rm u} $ is global and scatters with $ S_{\R}(\b{\rm u}) \lesssim \delta_{sd}^4$.
\end{enumerate}
\end{theorem}	

\begin{corollary}[Existence of nonlinear profile]\label{ENP}
There exists at least one nonlinear profile associated to $ (\b{\rm v}_0, \{ t_n \}_{n=1}^{\infty}) $.
\end{corollary}

\begin{proof}
There are two cases to consider, either $ \vert t_{\infty} \vert < \infty $ or $ \vert t_{\infty} \vert = \infty $.

{\bf Case 1}. If $ \vert t_{\infty} \vert < \infty $, set $ \b{\rm u}_0 = \b{\rm v}(t_{\infty}, x) = \verb"S"(t_{\infty}) \b{\rm v}_0 $ and solve the initial problem on any time interval $ I \ni t_{\infty} $ with $ \b{\rm u} \vert_{t = t_{\infty}} = \b{\rm u}_0 $, then
\begin{equation*}
\begin{aligned}
& ~ \lim_{n \to \infty} \Vert \b{\rm u}(t_n, \cdot) - \b{\rm v}(t_n, \cdot) \Vert_{{\rm\dot H}^1(\R^6)} \\
\le & ~ \lim_{n \to \infty} \Vert \verb"S"(t_n - t_{\infty}) \b{\rm u}_0 - \verb"S"(t_n) \b{\rm v}_0 \Vert_{{\rm\dot H}^1(\R^6)} + \lim_{n \to \infty} \left\Vert \int_{t_{\infty}}^{t_n} \verb"S"(t_n - s) \b{\rm f}(\b{\rm u}(s)) {\rm d}s \right\Vert_{{\rm\dot H}^1(\R^6)} \\
= & ~ \lim_{n \to \infty} \Vert \verb"S"(t_n) \b{\rm v}_0 - \verb"S"(t_n) \b{\rm v}_0 \Vert_{{\rm\dot H}^1(\R^6)} + \lim_{n \to \infty} \left\Vert \int_{t_{\infty}}^{t_n} \verb"S"(t_n - s) \b{\rm f}(\b{\rm u}(s)) {\rm d}s \right\Vert_{{\rm\dot H}^1(\R^6)} \\
= & ~ 0.
\end{aligned}
\end{equation*}

{\bf Case 2}. If $ \vert t_{\infty} \vert = \infty $, W. L. O. G., we can assume $ t_{\infty} = +\infty $ and there exists $ N \gg 1 $ such that
\[
\Vert \verb"S"(t) \b{\rm v}_0 \Vert_{\mathcal{S}([t_n, +\infty), {\rm\dot H}^1(\R^6))} \le \delta_{sd}, \qquad \forall ~ n \ge N,
\]
where $ \delta_{sd} $ is from Theorem \ref{sd}.
Solve the integral equation on time interval $ [t_n, +\infty) $
\[
\b{\rm u}(t) = \verb"S"(t) \b{\rm v}_0 + i \int_{t}^{+\infty} \verb"S"(t-s) \b{\rm f}(\b{\rm u}(s)) {\rm d}s,
\]
and one can obtain a solution $ \b{\rm u} $ to \eqref{NLS system} with terminal data $ \b{\rm v}_0 $.
Then
\begin{equation*}
\begin{aligned}
\lim_{n \to \infty} \Vert \b{\rm u}(t_n, \cdot) - \b{\rm v}(t_n, \cdot) \Vert_{{\rm\dot H}^1(\R^6)}
\le & ~  \lim_{n \to \infty} \left\Vert \int_{t_n}^{+\infty} \verb"S"(t_n - s) \b{\rm f}(\b{\rm u}(s)) {\rm d}s \right\Vert_{{\rm\dot H}^1(\R^6)} \\
\le & ~ C \lim_{n \to \infty} \Vert \nabla \b{\rm f}(\b{\rm u}) \Vert_{\mathcal{S}^{\prime}((t_n, +\infty), {\rm L}^2)(\R^6)}
= 0.
\end{aligned}
\end{equation*}

Collect the two cases above, the proof is completed.
\end{proof}

\begin{corollary}[Uniqueness of nonlinear profile]\label{UNP}
There exists at most one nonlinear profile associated to $ (\b{\rm v}_0, \{ t_n \}_{n=1}^{\infty}) $.
\end{corollary}

\begin{proof}
We prove it by contradiction and suppose $ \b{\rm u}^{(1)} $ and $ \b{\rm u}^{(2)} $ are both nonlinear profile associated to $ (\b{\rm v}_0, \{ t_n \}_{n=1}^{\infty}) $.

{\bf Case 1}. If $ \vert t_{\infty} \vert < \infty $, then the uniqueness of nonlinear profile can be obtained from Theorem \ref{sd}, Duhamel's formula and
\begin{equation*}
\begin{aligned}
& ~ \lim_{n \to \infty} \Vert \b{\rm u}^{(1)}(t_n, \cdot) - \b{\rm u}^{(2)}(t_n, \cdot) \Vert_{{\rm\dot H}^1(\R^6)} \\
\le & ~ \lim_{n \to \infty} \Vert \b{\rm u}^{(1)}(t_n, \cdot) - \b{\rm v}(t_n, \cdot) \Vert_{{\rm\dot H}^1(\R^6)} + \Vert \b{\rm u}^{(1)}(t_n, \cdot) - \b{\rm v}(t_n, \cdot) \Vert_{{\rm\dot H}^1(\R^6)} = 0.
\end{aligned}
\end{equation*}

{\bf Case 2}. If $ \vert t_{\infty} \vert = \infty $, W. L. O. G., we can assume $ t_{\infty} = +\infty $.
The fact that $ \Vert \nabla \b{\rm u}^{(i)} \Vert_{S(I, {\rm L}^2(\R^6))} \le \infty, i = 1, 2 $ implies that for any $ \varepsilon > 0 $, there exists $ N = N(\varepsilon) > 0 $ such that
\[
\Vert \nabla \b{\rm u}^{(i)} \Vert_{S((t_n, +\infty), {\rm L}^2(\R^6))} < \varepsilon, \qquad \forall ~ n \ge N.
\]
By theorem \ref{sd}, we have, for $ M \gg N $
\[
\sup_{t \in (t_N, t_M)} \Vert \nabla \b{\rm u}^{(1)}(t, \cdot) - \nabla \b{\rm u}^{(2)}(t, \cdot) \Vert_{{\rm L}^2(\R^6)} \le C \Vert \nabla \b{\rm u}^{(1)}(t_M, \cdot) - \nabla \b{\rm u}^{(2)}(t_M, \cdot) \Vert_{{\rm L}^2(\R^6)},
\]
which means that $ \b{\rm u}^{(1)} \equiv \b{\rm u}^{(2)} $ on time interval $ (t_N, +\infty) $ and thus on $ I $.
\end{proof}

We can establish the following perturbation theorem by similar argument.
For more details, one may refer to \cite{Gao2021, Tao2005}.	

\begin{proposition}[Stability, \cite{Gao2021}]\label{pt}
Given $ M \gg 1 $, there exist $ \varepsilon = \varepsilon (M) \ll 1 $ and $ L = L (M) \gg 1 $ such that the following holds.
Let $ \b{\rm u} = \b{\rm u}(t, x) \in {\rm\dot H}^1(\mathbb{R}^6) $ for all $ t $ and solve
\begin{equation*}
i \partial_{t} \b{\rm u} + A \b{\rm u} =  \b{\rm f}(\b{\rm u}) .
\end{equation*}
Let $ \tilde{\b{\rm u}} = \tilde{\b{\rm u}}(t, x) \in {\rm\dot H}_x^1(\mathbb{R}^6) $ for all $ t $ and define
\begin{equation*}
\b{\rm e} := i\partial_{t} \tilde{\b{\rm u}} + A \tilde{\b{\rm u}} - \b{\rm f}(\tilde{\b{\rm u}}).
\end{equation*}
If
\begin{equation*}
\begin{aligned}
& \Vert \tilde{\b{\rm u}} \Vert_{\mathcal{S}({\rm\dot H}^1(\mathbb{R}^6))} \le M, \qquad
\Vert \nabla \b{\rm e} \Vert_{\mathcal{S}^\prime({\rm L}^2(\mathbb{R}^6))} \le \varepsilon, \\
& \Vert \verb"S"(t-t_0) ( \b{\rm u}_0 - \tilde{\b{\rm u}}(t_0) ) \Vert_{\mathcal{S}({\rm\dot H}^1(\mathbb{R}^6))} \le \varepsilon,
\end{aligned}
\end{equation*}
then
\begin{equation*}
\Vert \b{\rm u} \Vert_{\mathcal{S}({\rm\dot H}^1(\mathbb{R}^6))} \le L.
\end{equation*}
\end{proposition}	
	
\subsection{Profiles decomposition}
In this section, we first introduce the linear profiles decomposition for the ${\rm\dot H}^1(\mathbb{R}^6)$ Sobolev function. After that, we will discuss the inverse Strichartz inequality and use it to establish the linear profiles decomposition for the Schr\"odinger flow $\verb"S"(t) $.

\begin{theorem}[Linear profiles decomposition for Sobolev functions \cite{Ja}]\label{Profile-K}
Let $ \b{\rm \phi}_n = ( \phi^1_n, \phi^2_n, \phi^3_n ) $ be an uniformly bounded sequence in $ {\rm\dot H}^1(\mathbb{R}^6) $ with $ \Vert \b{\rm \phi}_n \Vert \le A $ for any $ 1 \le n < +\infty $.
Then there exists $ J^* \in \{ 1, 2, \cdot\cdot\cdot \} \cup \{ \infty \} $ such that for each finite $ 1 \le J \le J^* $, we have the decomposition:
\[
\b{\rm \phi}_n(x) = \sum_{j=1}^J(\lambda_n^j)^{-\frac12}\b{\rm \phi}^j\Big(\frac{x-x_n^j}{\lambda_n^j}\Big)+\b{\rm r}_n^J(x), \qquad 1 \le J \le J^*.
\]
where $ \b{\rm r }^J_n := ( r^{1,J}_n, r^{2,J}_n, r^{3,J}_n ) $ in $ {\rm\dot H}^1(\mathbb{R}^6) $.
The decomposition has the following properties:
\begin{equation}
\limsup_{J\to J^*}\limsup_{n\to\infty}\|\b{\rm r}_n^J\|_{\rm L_x^6}=0,
\end{equation}
\begin{equation}
\limsup_{n\to\infty}\Big|\|\phi_n^{1}\|_{\rm\dot H^1}^2-\sum_{j=1}^J\|\phi^{j,1}\|_{\rm\dot H^1}^2-\|r_n^{J,1}\|_{\rm\dot H^1}^2\Big|=0,
\end{equation}
\begin{equation}
\limsup_{n\to\infty}\Big|\|\phi_n^{2}\|_{\rm\dot H^1}^2-\sum_{j=1}^J\|\phi^{j,2}\|_{\rm\dot H^1}^2-\|r_n^{J,2}\|_{\rm\dot H^1}^2\Big|=0,
\end{equation}
\begin{equation}
\limsup_{n\to\infty}\Big|\|\phi_n^{3}\|_{\rm\dot H^1}^2-\sum_{j=1}^J\|\phi^{j,3}\|_{\rm\dot H^1}^2-\|r_n^{J,3}\|_{\rm\dot H^1}^2\Big|=0.
\end{equation}

\begin{equation}
\liminf_{n\to\infty}\Big[\frac{|x_n^l-x_n^k|^2}{\lambda_n^l\lambda_n^k}
+\frac{\lambda_n^k}{\lambda_n^l}  +\frac{\lambda_n^l}{\lambda_n^k}\Big]=\infty, \qquad \forall ~ l \ne k,
\end{equation}
\begin{equation}
(\lambda_n^j)^{\frac12}\b{\rm r}_n^J(\lambda_n^j+x_n^j)\to0,\ weakly\ in\ \rm\dot H^1,
\end{equation}
\begin{align}
\limsup_{n\to\infty}\liminf_{n\to\infty}\big[V(\b{\rm \phi}_n)-\sum_{j=1}^{J}V(T_{\lambda_{n}^{j}}\b{\rm\phi}^j)-V(\b{\rm r}_n^J)\big]=0. \label{jianjin-P}
\end{align}
\end{theorem}
	
\begin{remark}
We use the  $ \mathcal{T}_\lambda \b{\rm u}$ to denote $(T_\lambda u^1, T_\lambda u^2, T_\lambda u^3)^T $ for $ \b{\rm u} = (u^1, u^2, u^3 )^T$, 	where $ T_{\lambda}f(x) := \frac1{\lambda^2} f(\frac{x - x^j_n}\lambda) $ and so $ T_{\lambda}^{-1}f(x) := \lambda^2 f(\lambda x + x^j_n) $.	
\end{remark}

Now, we introduce the refined Strichartz estimate, which is very powerful in the proof of the inverse Strichartz inequality.

\begin{lemma}[Refined Strichartz estimate]\label{RSE}
For any function $ \b{\rm g} \in {\rm\dot H}^1(\mathbb{R}^6) $, we have
\begin{equation}
\left\Vert \verb"S"(t) \b{\rm g} \right\Vert_{{\rm L}_{t, x}^4(\mathbb{R} \times \mathbb{R}^6)}
\lesssim \Vert \b{\rm g} \Vert_{{\rm\dot H}^1(\mathbb{R}^6)}^{\frac12} \sup_{N \in 2^{\mathbb{Z}}} \left\Vert \verb"S"(t) P_{N} \b{\rm g} \right\Vert_{{\rm L}_{t, x}^4(\mathbb{R} \times \mathbb{R}^6)}^{\frac12}.
\end{equation}
\end{lemma}

\begin{proposition}[Inverse Strichartz inequality]\label{ISI}
If the sequence $ \{ \b{\rm g}_n \}_{n=1}^\infty \subset {\rm\dot H}^1(\mathbb{R}^6) $ satisfies
\begin{equation*}
\lim_{n \to \infty} \Vert \b{\rm g}_n \Vert_{{\rm\dot H}^1(\mathbb{R}^6)} = A < \infty, \quad \text{and} \quad \lim_{n \to \infty} \left\Vert \verb"S"(t) \b{\rm g}_n \right\Vert_{{\rm L}_{t, x}^4(\mathbb{R} \times \mathbb{R}^6)} = \eps > 0.
\end{equation*}
Then there exists a subsequence of $ \{ n \} ($still denoted by $ \{ n \})$, $ \b{\rm \phi} \in {\rm\dot H}^1(\R^6) $, $ \{ \lambda_n \}_{n=1}^\infty \subset (0, \infty) $ and $ \{ (t_n, x_n) \}_{n=1}^\infty \subset \mathbb{R} \times \mathbb{R}^6 $ such that

\begin{equation}
\lambda_n^2 \left[ \verb"S"(t_n) \b{\rm g}_n \right] (\lambda_n x + x_n) \rightharpoonup \b{\rm \phi} (x) \quad \text{weakly in $ {\rm\dot H}^1(\mathbb{R}^6) $},
\end{equation}

\begin{equation}
\liminf_{n \to \infty} \left\{ \Vert \b{\rm g}_n \Vert_{{\rm\dot H}^1(\mathbb{R}^6)}^2 - \Vert \b{\rm g}_n - \b{\rm \phi}_n \Vert_{{\rm\dot H}^1(\mathbb{R}^6)}^2 \right\}
= \Vert \b{\rm \phi} \Vert_{{\rm\dot H^1(\mathbb{R}^6)}}^2
\gtrsim A^{-10} \eps^{12},
\end{equation}

\begin{equation}
\liminf_{n \to \infty} \left\{ \left\Vert \verb"S"(t) \b{\rm g}_n \right\Vert_{{\rm L}_{t, x}^4(\mathbb{R} \times \mathbb{R}^6)}^4  - \left\Vert \verb"S"(t) ( \b{\rm g}_n - \b{\rm \phi}_n ) \right\Vert_{{\rm L}_{t, x}^4(\mathbb{R} \times \mathbb{R}^6)}^4 \right\}
\gtrsim A^{-20} \eps^{24},
\end{equation}

\begin{equation}
\liminf_{n \to \infty} \left\{ \Vert \b{\rm g}_n \Vert_{{\rm L}^3(\mathbb{R}^6)}^3 - \Vert \b{\rm g}_n - \b{\rm \phi}_n \Vert_{{\rm L}^3(\mathbb{R}^6)}^3 - \left\Vert \verb"S" \left( - \frac{t_n}{\lambda_n^2} \right) \b{\rm \phi} \right\Vert_{{\rm L}^3(\mathbb{R}^6)}^3 \right\} = 0,
\end{equation}	
where
\begin{equation}
\b{\rm \phi}_n (x) := \frac1{\lambda_n^2} \left[ \verb"S" \left( -\frac{t_n}{\lambda_n^2} \right) \b{\rm \phi} \right] \left( \frac{x - x_n}{\lambda_n} \right).
\end{equation}
\end{proposition}

Combining the refine Strichartz estimate and  inverse Strichartz inequality, we can establish the linear profiles decompositions for the the Schr\"odinger flow $\verb"S"(t) $.

\begin{theorem}[Linear profiles decomposition for the Schr\"odinger flow]\label{lpd}
Let $ \b{\rm \phi}_n = ( \phi^1_n, \phi^2_n, \phi^3_n ) $ be an uniformly bounded sequence in $ {\rm\dot H}^1(\mathbb{R}^6) $ with $ \Vert \b{\rm \phi}_n \Vert \le A $ for any $ 1 \le n < +\infty $.
Then there exists $ J^* \in \{ 1, 2, \cdots \} \cup \{ \infty \} $ such that for each finite $ 1 \le J \le J^* $, we have the decomposition:
	
\begin{equation}
\b{\rm \phi}_n = \sum_{j=1}^J \mathcal{T}_{\lambda^j_n} \left[ \verb"S" \left( t^j_n\right) \b{\rm \phi}^j \right] + \b{\rm \Phi}^J_n,
\label{pro}
\end{equation}
where $ \b{\rm \Phi}^J_n := ( \Phi^{1,J}_n, \Phi^{2,J}_n, \Phi^{3,J}_n ) $ in $ {\rm\dot H}^1(\mathbb{R}^6) $.

In fact, the $ \b{\rm \Phi}^J_n := ( \Phi^{1,J}_n, \Phi^{2,J}_n, \Phi^{3,J}_n ) $ in $ {\rm\dot H}^1(\mathbb{R}^6) $ has some properties as follow:
\begin{equation}
\lim_{J \to J^*} \limsup_{n \to \infty} \left\Vert \verb"S"(t) \b{\rm \Phi}^J_n \right\Vert_{{\rm L}_{t, x}^4(\mathbb{R} \times \mathbb{R}^6)}
= 0,
\label{remainder1}
\end{equation}

\begin{equation}
\verb"S"(-t^J_n) \left[ \mathcal{T}_{\lambda^J_n}^{-1} \b{\rm \Phi}^J_n \right] \rightharpoonup 0 \qquad \text{weakly in $ {\rm\dot H}^1(\mathbb{R}^6) $},
\label{remainder2}
\end{equation}

\begin{equation}
\left\Vert \phi^1_n \right\Vert_{{\dot H}^1(\mathbb{R}^6)}^2
= \sum_{j=1}^J \left\Vert \phi^{1,j} \right\Vert_{{\dot H}^1(\mathbb{R}^6)}^2 + \left\Vert \Phi^{1,J}_n \right\Vert_{{\dot H}^1(\mathbb{R}^6)}^2 + o_n (1),
\label{ape1}
\end{equation}

\begin{equation}
\left\Vert \phi^2_n \right\Vert_{{\dot H}^1(\mathbb{R}^6)}^2
= \sum_{j=1}^J \left\Vert \phi^{2,j} \right\Vert_{{\dot H}^1(\mathbb{R}^6)}^2 + \left\Vert \Phi^{2,J}_n \right\Vert_{{\dot H}^1(\mathbb{R}^6)}^2 + o_n (1),
\label{ape2}
\end{equation}

\begin{equation}
\left\Vert \phi^3_n \right\Vert_{{\dot H}^1(\mathbb{R}^6)}^2
= \sum_{j=1}^J \left\Vert \phi^{3,j} \right\Vert_{{\dot H}^1(\mathbb{R}^6)}^2 + \left\Vert \Phi^{3,J}_n \right\Vert_{{\dot H}^1(\mathbb{R}^6)}^2 + o_n (1).
\label{ape3}
\end{equation}

\begin{equation}
\left\Vert \phi^1_n \right\Vert_{L^3(\mathbb{R}^6)}^3
= \sum_{j=1}^J \left\Vert e^{\kappa_1it^j_n\Delta}\phi^{1,j} \right\Vert_{L^3(\mathbb{R}^6)}^3 + \left\Vert \Phi^{1,J}_n \right\Vert_{L^3(\mathbb{R}^6)}^3 + o_n (1),
\label{ape4}
\end{equation}

\begin{equation}
\left\Vert \phi^2_n \right\Vert_{L^3(\mathbb{R}^6)}^3
= \sum_{j=1}^J \left\Vert e^{\kappa_2it^j_n\Delta}\phi^{2,j} \right\Vert_{L^3(\mathbb{R}^6)}^3 + \left\Vert \Phi^{2,J}_n \right\Vert_{L^3(\mathbb{R}^6)}^3 + o_n (1),
\label{ape5}
\end{equation}

\begin{equation}
\left\Vert \phi^3_n \right\Vert_{L^3(\mathbb{R}^6)}^3
= \sum_{j=1}^J \left\Vert e^{\kappa_3it^j_n\Delta}\phi^{3,j} \right\Vert_{L^3(\mathbb{R}^6)}^3 + \left\Vert \Phi^{3,J}_n \right\Vert_{L^3(\mathbb{R}^6)}^3 + o_n (1).
\label{ape6}
\end{equation}

Besides, for any $ j \ne l $, we have the almost orthogonal conditions as follows:	

\begin{equation}
\frac{\lambda^j_n}{\lambda^l_n} + \frac{\lambda^l_n}{\lambda^j_n} + \frac{\vert x^j_n - x^l_n \vert^2}{\lambda^j_n \lambda^l_n} + \frac{\vert t^j_n (\lambda^j_n)^2 - t^l_n (\lambda^l_n)^2 \vert}{\lambda^j_n \lambda^l_n} \to + \infty, \quad \forall ~ j \ne l \in \{ 1, 2, \cdot\cdot\cdot, J \}.
\label{asy}
\end{equation}

In addition, we may assume that  either $ t^j_n \equiv 0 $ or $ t^j_n \to \pm \infty $ for any $ 0 \le j \le J $.
Particularly, fixed $ 1 \le j_0 \le J $, there exists $ \alpha(j_0) > 0 $ such that
\begin{equation}
K\left( \b{\rm \phi}^{j_0} \right) \ge \alpha(j_0).
\label{H>0}
\end{equation}
\end{theorem}

In order to focus on the main topics and the core technique of this paper, we leave out the proofs of Lemma \ref{RSE}, Proposition \ref{ISI} and Theorem \ref{lpd} here, which are standard in \cite{Killip2008}.

\subsection{Virial identity}
In this section, we discuss some Virial-type identity, which plays an important role in the proof Theorem \ref{Sb} and \ref{Bu}.

\begin{proposition}(Virial identity) \label{vi}
Let $\b{\rm u}:=(u^1, u^2, u^3)$ be a solution to the \eqref{NLS system},  for any real valued function $ a \in C^\infty(\mathbb{R}^6) $, defining functions $V_{1}, V_{2}$ by
\begin{align}
V_{1}(t)&=\int_{\mathbb{R}^6} \left(  \kappa_2\kappa_3 \vert u^1 \vert^2 + \kappa_1\kappa_3\vert u^2 \vert^2 +\kappa_1\kappa_2 \vert u^3\vert^2 \right) a(x) {\rm d}x\\
V_{2}(t)&=2 \Im \int_{\mathbb{R}^6} \left(  \overline{u^1} \nabla u^1 + \overline{u^2} \nabla u^2  + \overline{u^3} \nabla u^3\right) \cdot \nabla a(x) {\rm d}x.
\end{align}
Then we have
\begin{align*}
\frac{{\rm d}}{{\rm d}t}V_{1}(t)=&2 \kappa_1\kappa_2\kappa_3 \Im \int_{\mathbb{R}^6} \left(  \overline{u^1} \nabla u^1 + \overline{u^2} \nabla u^2  + \overline{u^3} \nabla u^3\right) \cdot \nabla a(x) {\rm d}x\\
-&2(\kappa_{2}\kappa_{3}+\kappa_{1}\kappa_{3}-\kappa_{1}\kappa_{2})\Im\int_{\R^{6}}\overline{u^{1}u^{2}}u^{3}a(x)\rm dx.\\
\frac{{\rm d}}{{\rm d}t}V_{2}(t)&= \Re \int_{\mathbb{R}^6} \left( 4\kappa_1\overline{u^1_j} u^1_{k} + 4\kappa_2 \overline{u^2_j} u^2_{k}+ 4\kappa_3\overline{u^3_j}u^3_k \right) a_{jk}(x) {\rm d}x \\
&- \int_{\mathbb{R}^6} \left(\kappa_1 \vert u^1 \vert^2 + \kappa_2 \vert u^2 \vert^2 +\kappa_3 \vert u^3\vert^2 \right) \Delta \Delta a(x) {\rm d}x \\
&- 2 \Re \int_{\mathbb{R}^6} \overline{u^1u^2} u^3 \Delta a(x) {\rm d}x.
\end{align*}		
\end{proposition}

Assuming the \eqref{NLS system} satisfies the mass-resonance and  set $a(x)=\vert x \vert^2$, a simple calculation shows that
\[
a_j(x):=\partial_{x_j}a = 2 x_j, \qquad \nabla a(x) = 2 x,
\]
and
\[
a_{jk}(x):=\partial_{x_i}\partial_{x_j}a = 2 \delta_{jk}, \qquad \Delta a(x) = 2d, \qquad \Delta \Delta a(x) = 0.
\]
Directly computing, we have the following identity
\begin{align}
V_{1}^{\prime\prime}(t) = 8\kappa_1\kappa_2\kappa_3 [ 2 K(\b{\rm u}) - 3 V(\b{\rm u}) ].
\label{I''}
\end{align}

%基态刻画
\section{Variational Analysis}\label{VA}
In this section, we will establish the analysis theory of the ground state and use this to give the energy trapping about the ground state.

\subsection{Ground state}\label{ground}
The existence of ground states has been shown in \cite{Meng2021, Meng2023, Noguera2020} for $ 1 \le d \le 5 $ and in \cite{Noguera2022} for $ d = 6 $.
For the sake of completeness, we still give the proof of existence of ground state in energy-critical case $ d = 6 $ here.
We first recall some basic properties of ground state for the \eqref{NLS system}, which is similar to the classical \eqref{NLS}. Then, we use it to get the sharp Gargliardo-Nirenberg inequality.

\begin{proposition} \label{ground state}
	We define the nonnegative functional $J(\b{\rm X}) := \left[ K(\b{\rm X}) \right]^3 \left[ V(\b{\rm X}) \right]^{-2}$ on $ \Xi :=\left\{ \b{\rm X} \in {\rm\dot H}^1(\R^6)  ~ \big\vert ~ V(\b{\rm X}) \neq 0 \right\} $.  Then there exists non-negative radial function $ \b{\rm W}=(\phi_{1},\phi_{2},\phi_{3})^{T} \in \Xi $ that minimize the function $J(\b{\rm X})$. Furthermore, the minimizer $\b{\rm \Phi}$ is the solution of \eqref{gs}. Besides, we use the notation $\mathcal{G}$  to denote all the non-negative radial solutions.	
%This  functional can attain the minimal $ J_{\min} $ at $\b{\rm X}= \b{\rm G}=(\Phi^1, \Phi^2, \Phi^3 ) \in \C^3$.  The expression has the form of $ (\Phi^1, \Phi^1, \Phi^1) = (e^{i\theta_1} m \phi_1(nx), e^{i\theta_2} m \phi_2(nx), e^{i\theta_3} m \phi_3(nx) )$, where $ m > 0 $, $ n > 0 $, $ \theta_i \in \mathbb R  (i=1, 2, 3),$ and $ \b{\rm W} := (\phi_1, \phi_2, \phi_3)^T \in \Xi $ is the non-negative real-valued solution of \eqref{gs}.  The function $\b{\rm W}$ is defined a ground state satisfying $J(\b{\rm W}) = J_{min}$. Besides, we use the notation $\mathcal{G}$  to denote the set of all ground states.
\end{proposition}
\begin{proof}
	{\bf Step 1.}  We show that the minimum of $J(\b{\rm X})$ can be attained. Suppose that the non-zero function sequence $\{\mathbf{X}_n\}$ is the minimal sequence of the functional $J$, i.e.
	$$
	\lim_{n\to\infty} J(\mathbf{X}_n)=\inf\{J(\mathbf{X}): \mathbf{X}\in \Xi\}.
	$$
	By the Schwartz symmetrical rearrangement lemma, without loss of generality, we assume $\mathbf{X}_n$ are non-negative radial.
	
	Note that for any $ \mathbf{X}\in \Xi$, $ \mu > 0 $, and $ \nu > 0 $, we have
	\begin{equation}\label{mhr}
	\left\{
	\begin{aligned}
	& K(\mu\mathbf{X}(\nu ~ \cdot ~)) = \mu^2 \nu^{-4} K(\mathbf{X}), \\
	& V(\mu\mathbf{X}(\nu ~ \cdot ~)) = \mu^3 \nu^{-6} V(\mathbf{X}).
	\end{aligned}
	\right.
	\end{equation}
	By the definition of $ J(\mathbf{X}) $, it is easy to see $ J(\mu\mathbf{X}(\nu ~ \cdot ~)) = J(\mathbf{X}) $. Denote
	$$
	\mathbf{v}_n(x) = \mu_0 \mathbf{X}_n(\nu_0 x), \quad \text{for any} ~ \mu_0 = [K(\mathbf{X}_n)]^{x} ~ \text{and} ~ \nu_0 = [K(\mathbf{X}_n)]^{y},
	$$
	as long as the real pair $ (x, y) $ satisfies $ 2x - 4y + 1 = 0 $.
	On one hand, the kinetic energy of $ \mathbf{v}_n $ have been unitized, i.e.
	$$
	K(\mathbf{v}_n) \equiv 1.
	$$
	On the other hand, $ \mathbf{v}_n $ also retain non-negative radial at the same tine and satisfies
	$ J(\mathbf{v}_n)=J(\mathbf{X}_n) $, which implies
	$$
	\lim_{n\to\infty} J(\mathbf{v}_n)=\inf\{J(\mathbf{u}): \mathbf{u} \in \Xi \}=J_{\min}.
	$$
	
	We observe that $ \{ \mathbf{v}_n \}_{n=1}^\infty \subset {\rm\dot H}^1_{\rm rad}(\R^6) $ is bounded sequence. So according to Theorem \ref{Profile-K},  we get
	$$V(\b{\rm v}_{n})=\sum_{j=1}^{J}V(T_{\lambda_{n}^{j}}\b{\rm\phi}^j)+V(\b{\rm r}_n^J)\leq J_{\min}^{-\frac12}\Big[\sum_{j=1}^{J}K(\b{\rm \phi}^{j})^{\frac32}+K(\b{\rm r}_{n}^{J})^{\frac32}\Big].$$
	If $J^{*}>1$, by Theorem \ref{Profile-K}, we can obtain that
	$$|V(\b{\rm v}_{n})|< J_{\min}^{-\frac12}$$
	which contradict to the definition of $J_{\min}$. Thus we have $J^{*}=1$ and
	$$\mathcal{T}_{\lambda_{n}^{1}}^{-1}\b{\rm v}_{n}\to\b{\rm v}^{\star}\ \ as\ \ n\to\infty\ \ in\ \ \rm\dot H^{1}.$$
	Hence $J_{\min}=J(\b{\rm v}^{\star})$.
	By above argument, we show that the minimal of $ J $ are attained at $ \mathbf{v}^\star $.

{\bf Step 2.}
We consider the variational derivatives of $ K $ and $ V $: fix $ \b{\rm v}^{*}=(\Phi_1, \Phi_2, \Phi_3)^T $, for any $ \mathbf{w}=(w_1, w_2, w_3)^T \in {\rm\dot H}^1(\R^6) $,
\begin{align*}
\left.\frac{{\rm d}}{{\rm d}h}\right|_{h=0} K(\b{\rm v}^{*} + h \mathbf{w})
= & ~ \Re \int_{\R^6} (-\kappa_1\Delta\Phi_1, -\kappa_2\Delta\Phi_2, -\kappa_3\Delta\Phi_3 ) \cdot \overline{(w_1, w_2, w_3)} {\rm d}x \\
\left.\frac{{\rm d}}{{\rm d}h}\right|_{h=0} V(\b{\rm v}^{*} + h \mathbf{w})
= & ~ \Re \int_{\R^6} (\bar{\Phi}_2\Phi_3, \bar{\Phi}_1\Phi_3, \Phi_1\Phi_2) \cdot \overline{(w_1, w_2, w_3)} {\rm d}x,
\end{align*}
Using the extremum principle, we have, for any $ \mathbf{w} = (w_1, w_2, w_3)^T \in \dot{\rm H}^1(\R^6) $,
$$
\left.\frac{{\rm d}}{{\rm d}h}\right|_{h=0}J(\b{\rm v}^{*}+h\mathbf{w})=0.
$$
A direct calculation show that
\begin{align*}
0 = & ~ \left.\frac{{\rm d}}{{\rm d}h}\right|_{h=0}\left( [K(\b{\rm v}^{*}+h\mathbf{w})]^3 [V(\b{\rm v}*+h\mathbf{w})]^{-2} \right) \\
= & ~ 3 \left( K(\b{\rm v}^{*})]^2 [V(\b{\rm v}^{*})]^{-2} \right) \left(\Re \int_{\R^6} (-\kappa_1\Delta\Phi_1, -\kappa_2\Delta\Phi_2, -\kappa_3\Delta\Phi_3 ) \cdot \overline{(w_1, w_2, w_3)} {\rm d}x \right) \\
& ~ -2 \left( [K(\b{\rm v}^{*})]^3 [V(\b{\rm v}^{*})]^{-3} \right) \left(\Re \int_{\R^6} (\Phi_2\Phi_3, \Phi_1\Phi_3, \Phi_1\Phi_2) \cdot \overline{(w_1, w_2, w_3)} {\rm d}x \right).
\end{align*}
Multiplying both sides of the above formula by $  [K(\b{\rm v}^*)]^{-3} [V(\b{\rm v}^*)]^2 $, we find
$$
\begin{aligned}
0 = & ~ 3[K(\b{\rm v}^{*})]^{-1}\Re \int_{\R^6} (-\kappa_1\Delta\Phi_1, -\kappa_2\Delta\Phi_2, -\kappa_3\Delta\Phi_3 ) \cdot \overline{(w_1, w_2, w_3)} {\rm d}x\\
& ~ -2[V(\b{\rm v}^{*})]^{-1}\Re \int_{\R^6} (\Phi_2\Phi_3, \Phi_1\Phi_3, \Phi_1\Phi_2) \cdot \overline{(w_1, w_2, w_3)} {\rm d}x.
\end{aligned}
$$
By the arbitrariness of $ \mathbf{w} = (w_1, w_2, w_3)^T \in \C^2 $, we have
\begin{equation}
\left\{
\begin{aligned}
&-3[K(\Phi_1, \Phi_2, \Phi_3)]^{-1} \kappa_1 \Delta\Phi^1 - 2[V(\Phi_1, \Phi_2, \Phi_3)]^{-1} \Phi_2 \Phi_3 = 0, \\
&-3[K(\Phi_1, \Phi_2, \Phi_3)]^{-1} \kappa_2 \Delta\Phi^2 - 2[V(\Phi_1, \Phi_2, \Phi_3)]^{-1} \Phi_1 \Phi_3 = 0, \\
&-3[K(\Phi_1, \Phi_2, \Phi_3)]^{-1} \kappa_3 \Delta\Phi^3 - 2[V(\Phi_1, \Phi_2, \Phi_3)]^{-1} \Phi_1 \Phi_2 = 0.
\end{aligned}
\right.
\label{PP}
\end{equation}
And we will find that \eqref{PP} is equivalent to
\begin{equation}
\left\{
\begin{aligned}
&-\kappa_1\Delta\Phi_1 = \alpha \Phi_2\Phi_3, \\
&-\kappa_2\Delta\Phi_2 = \alpha \Phi_1\Phi_3, \\
&-\kappa_3\Delta\Phi_3 = \alpha \Phi_1\Phi_2,
\end{aligned}
\right.
\label{complex}
\end{equation}
if we denote $ \alpha := \frac{3K(\Phi_1, \Phi_2, \Phi_3)}{2V(\Phi_1, \Phi_2, \Phi_3)} $ to be a real constant.
For the real functions
\begin{equation}
(\phi_1(x), \phi_2 (x), \phi_3 (x)) := (\alpha^a  \Phi_1(\alpha^b x) , \alpha^a  \Phi_2(\alpha^b x), \alpha^a  \Phi_3(\alpha^b x)),
\label{scale}
\end{equation}
where $ a - 2b - 1 = 0 $, we have
$$
\left\{
\begin{aligned}
&-\kappa_1\Delta\phi_1 = \phi_2\phi_3, \\
&-\kappa_2\Delta\phi_2 = \phi_1\phi_3,\\
&-\kappa_3\Delta\phi_3 = \phi_1\phi_2.
\end{aligned}
\right.
\qquad x \in \R^6.
$$
Thus, we complete the proof. Generally, the non-negative radial minimizer $ \mathbf{W}(x) = (\phi_1(x), \phi_2(x), \phi_3(x))^T \ne \b{\rm 0} $ is called the \emph{ground state}.
\end{proof}

\begin{remark}
We give a different proof than that of \cite{Noguera2022} where the authors use the concentrate-compactness principle to illustrate the attainment of minimizer of $ J $.
Here, we use lower semicontinuity in the way of 2a.
\end{remark}

Now, we have the existence but no uniqueness of ground state, interested readers can find the difficulties in \cite{Hayashi2011} and \cite{Hayashi2022}.
Although the following results can be deduced by usual approaches, we state and prove them here.

\begin{proposition}\label{KV}
For the kinetic and potential energy of the ground state  $ \b{\rm W} := (\phi_1, \phi_2, \phi_3)^T \in \mathcal{G} $, we have $ K(\b{\rm W}) : V(\b{\rm W}) = 3 : 2.$
\end{proposition}

\begin{proof}
By scaling argument, we can have following identity:
\begin{equation*}
E \left( \lambda^{\alpha} \b{\rm W}(\lambda^\beta ~ \cdot ~) \right) = \lambda^{2\alpha - 4\beta}K(\b{\rm W}) - \lambda^{3\alpha - 6\beta}V(\b{\rm W}), \qquad \forall ~ \lambda \in (0, \infty).
\end{equation*}	
Using Hopf-variational and letting $ \lambda = 1 $, we get
\[
\begin{aligned}
0 = & ~ \Re \int_{\mathbb{R}^6} \left( -\kappa_1\Delta \phi_1 -\phi_2 \phi_3 \right) \cdot \overline{ \left( \frac{\rm d}{{\rm d}\lambda} \Big|_{\lambda=1} \lambda^{\alpha} \phi_1(\lambda^\beta x) \right) } \\
& \quad \quad + \left( -\kappa_2\Delta \phi_2 -\phi_1\phi_3 \right) \cdot \overline{ \left( \frac{\rm d}{{\rm d}\lambda} \Big|_{\lambda=1} \lambda^{\alpha} \phi_2(\lambda^\beta x) \right) }   \\
& \quad \quad + \left( -\kappa_3\Delta \phi_3 -\phi_1\phi_2 \right) \cdot \overline{ \left( \frac{\rm d}{{\rm d}\lambda} \Big|_{\lambda=1} \lambda^{\alpha} \phi_3(\lambda^\beta x) \right) } {\rm d}x \\
= & ~ (2\alpha - 4\beta) K(\b{\rm W}) - (3\alpha - 6\beta) V(\b{\rm W}) \\
= & ~ [2K(\b{\rm W}) - 3V(\b{\rm W})] \alpha - [4K(\b{\rm W}) - 6V(\b{\rm W})] \beta, \quad \forall ~ \alpha \in \mathbb{R}, ~ \beta \in \R.
\end{aligned}
\]	
This yields the desired result.	
\end{proof}

\begin{proposition}(Sharp Sobolev inequality)\label{GN}
For any $ \b{\rm g} \in {\rm\dot H}^1(\mathbb{R}^6) $,
\begin{equation}
V(\b{\rm g}) \le C_{GN} [K(\b{\rm g})]^{\frac32},
\label{GNi}
\end{equation}
where
$ C_{GN} = J_{\min}^{-\frac12} > 0 $ is a constant.
Furthermore, $ V(\b{\rm W}) = C_{GN} [K(\b{\rm W})]^{\frac32} $ for any $ \b{\rm W} \in \mathcal{G} $.	
\end{proposition}

\begin{proof}
By Proposition \ref{ground state}, for any $\b{\rm g}\in \Xi $, we obtain that
	\[
	[K(\b{\rm g})]^3 [V(\b{\rm g})]^{-2} = J(\b{\rm g}) \ge J_{\min},
	\]
	which implies the \eqref{GNi}.
And it is easy to check that $ V(\b{\rm W}) = C_{GN} [K(\b{\rm W})]^{\frac32} $ for any $ \b{\rm W} \in \mathcal{G} $.
\end{proof}

\begin{proposition}[Coercivity of energy]\label{coer}
	For $ 0 < \rho < 1 $.
	Given initial data $ \b{\rm u}_0 \in {\rm\dot H}^1(\R^6) $ and assuming $ E(\b{\rm u}_0) \le (1 - \rho) E(\b{\rm W}) $ to \eqref{NLS system}, if $ K(\b{\rm u}_0) < K(\b{\rm W}) $, then there exists $ \rho^{\prime\prime} = \rho^{\prime\prime} (\rho) > 0 $ such that
	\begin{equation}
	2 K(\b{\rm u}) - 3 V(\b{\rm u}) \ge \rho^{\prime\prime} K(\b{\rm u});
	\label{coe}
	\end{equation}
	if $ K(\b{\rm u}_0) > K(\b{\rm W}) $ then there exists $ \tilde\rho^{\prime\prime} = \tilde\rho^{\prime\prime} (\rho) > 0 $ such that
	\begin{equation}
	2 K(\b{\rm u}) - 3 V(\b{\rm u}) \le - \tilde\rho^{\prime\prime} K(\b{\rm u}).
	\label{coe-1}
	\end{equation}
\end{proposition}

\begin{proof}
	 By sharp Sobolev inequality \eqref{GNi}, we have
	\begin{equation}
	E(\b{\rm u}) = K(\b{\rm u}) - V(\b{\rm u}) \ge K(\b{\rm u}) - C_{GN} [K(\b{\rm u})]^{\frac32}.
	\label{f}
	\end{equation}
	Define $f(y) := y - C_{GN} y^{\frac32} $ for $ y = y(t) := K(\b{\rm u}(t)) \ge 0 $.
	Then $ f^\prime(y) = 1 - \frac32 C_{GN} y^{\frac12} $. If $ y \le y_0 := \frac4{9C_{GN}^2} $, we know $ f^\prime(y) \ge 0 $.
	Using Proposition \ref{KV} and \ref{GN}, we get $ C_{GN} = \frac{2}{3[K(\b{\rm W})]^{\frac12}}$ and $y_0 = K(\b{\rm W})$.
Thus, we have
	\begin{equation}
	f_{\max}(y) =  f(K(\b{\rm W})) = K(\b{\rm W}) - C_{GN} [K(\b{\rm W})]^{\frac32} = \frac13 K(\b{\rm W}) = E(\b{\rm W}).
	\label{fmax}
	\end{equation}
	For any time $ t \in I $, according to  \eqref{f}, \eqref{fmax}, and the conversation of energy, we obtain
	\begin{equation*}
	f(K(\b{\rm u}(t))) \le E(\b{\rm u}(t)) = E(\b{\rm u_0}) < E(\b{\rm W}) = f(K(\b{\rm W})),
	\end{equation*}
	which implies $ K(\b{\rm u}(t)) \ne K(\b{\rm W}) $ for any $ t \in I $.
	
	If $ K(\b{\rm u}_0) < K(\b{\rm W}) $, using the continuity of $ f(y(t))$ with respect to the variable $t$,  there exists $ \rho^\prime = \rho^\prime (\rho) > 0 $ such that
	\begin{equation}
	K(\b{\rm u}(t)) \le (1 - \rho^\prime) K(\b{\rm W}), \quad \forall ~ t \in I.
	\label{<}
	\end{equation}
	Similarly, if $ K(\b{\rm u}_0) > K(\b{\rm W}) $, there exists $ \tilde\rho^\prime = \tilde\rho^\prime (\rho) > 0 $ such that
	\begin{equation}
	K(\b{\rm u}(t)) \ge (1 + \tilde\rho^\prime) K(\b{\rm W}), \quad \forall ~ t \in I.
	\label{>}
	\end{equation}
	
	Moreover, we can find  $\rho^{\prime\prime} = \rho^{\prime\prime} (\rho^\prime) > 0 $ and $ \tilde\rho^{\prime\prime} = \tilde\rho^{\prime\prime} (\rho, \tilde\rho^\prime) > 0 $ such that
	\[
	2 - 3 \frac{V(\b{\rm u})}{K(\b{\rm u})} \ge 2 - 3 C_{GN}[K(\b{\rm u})]^{\frac12} = 2 - 2 \left( \frac{K(\b{\rm u})}{K(\b{\rm W})} \right)^{\frac12} \ge 2 \left[ 1 - \left( 1 - \rho^\prime \right)^{\frac12} \right] := \rho^{\prime\prime};
	\]
	and
	\[
	2 - 3 \frac{V(\b{\rm u})}{K(\b{\rm u})} = 3 \frac{E(\b{\rm u}_0)}{K(\b{\rm u})} - 1 \le 3 \frac{(1 - \rho)E(\b{\rm W})}{(1 + \rho^\prime)K(\b{\rm W})} - 1 = \frac{(1 - \rho)}{(1 + \rho^\prime)} - 1 := - \tilde\rho^{\prime\prime}.
	\]
	By the above argument, for $K(\b{\rm u}_{0})<K(\b{\rm W})$, we obtain
	\begin{equation*}
	2 K(\b{\rm u}) - 3 V(\b{\rm u}) \ge \rho^{\prime\prime} K(\b{\rm u});
	\label{coe''}
	\end{equation*}
	and for $K(\b{\rm u}_{0})>K(\b{\rm W})$, we have
	\begin{equation*}
	2 K(\b{\rm u}) - 3 V(\b{\rm u}) \le - \tilde\rho^{\prime\prime} K(\b{\rm u}).
	\label{coe'''}
	\end{equation*}
\end{proof}

\subsection{Energy trapping}\label{et}
Then we want to establish some coercivity estimates, which is useful for us to discuss the dynamics for the \eqref{NLS system}.

\begin{proposition}\label{E>0}
	Given $ \b{\rm v} \in {\rm\dot H}^1(\mathbb{R}^6) $, if $ K(\b{\rm v}) \le K(\b{\rm W}) $, then
	$
	E(\b{\rm v}) \ge 0.
	$
\end{proposition}

\begin{proof}
	Considering $ 0 \le K(\b{\rm v}) \le K(\b{\rm W}) $, we define $ f(y) = y - C_{GN} y^{\frac32} $ for $ 0 \le y \le K(\b{\rm W}) $.
	Then $ f^\prime(y) = 1 - \frac32 C_{GN} y^{\frac12} $.
	If $ y \le y_0 := \frac4{9C_{GN}^2} $, we know $ f^\prime(y) \ge 0 $.
	  We use the fact $ C_{GN} = \frac{V(\b{\rm W})}{[K(\b{\rm W})]^{\frac32}} $ and $ K(\b{\rm W}) : V(\b{\rm W}) = 2 : 3 $, which implies $$ C_{GN} = \frac{2}{3[K(\b{\rm W})]^{\frac12}} \quad \text{and}\quad  y_0 = K(\b{\rm W}) .$$
	
	Observing that $
	E(\b{\rm v}) = K(\b{\rm v}) - V(\b{\rm v}) \ge K(\b{\rm v}) - C_{GN} [K(\b{\rm v})]^{\frac32}.$
	Therefore, we obtain $ E(\b{\rm v}) \ge f(K(\b{\rm v})) \ge 0 $ if $ 0 \le K(\b{\rm v}) \le K(\b{\rm W}) $.
\end{proof}

According to Arithmetic-Geometry Mean-value inequality and Proposition \ref{GN}, we can establish the following energy trapping.
And we start with a new bound of potential energy.

\begin{lemma}\label{Ck}
For any $ \b{\rm v} \in {\rm\dot H}^1(\mathbb{R}^6) $, we have
\[
\vert V(\b{\rm v}) \vert \le C(\kappa_1, \kappa_2, \kappa_3) \left[ K(\b{\rm v}) \right]^{\frac32},
\]
where $ C(\kappa_1, \kappa_2, \kappa_3) = \sqrt{\frac8{27\kappa_1\kappa_2\kappa_3}} $.
\end{lemma}

\begin{proof}
Set $ \b{\rm v} =: (v^1, v^2, v^3)^T $.
By the definition of potential energy, H\"older inequality, and Sobolev embedding $ \dot{H}^1(\R^6) \hookrightarrow L^3(\R^6) $,
\begin{equation*}
\begin{aligned}
\vert V(\b{\rm v}) \vert = & ~ \left\vert \Re \int_{\mathbb{R}^6} \overline{v^1}\overline{v^2}v^3 {\rm d}x \right\vert \\
\le & ~ \int_{\mathbb{R}^6} \left\vert \overline{v^1}\overline{v^2}v^3 \right\vert {\rm d}x \\
\le & ~ \Vert v^1 \Vert_{L^3(\R^6)} \Vert v^2 \Vert_{L^3(\R^6)} \Vert v^3 \Vert_{L^3(\R^6)} \\
\le & ~ \Vert v^1 \Vert_{\dot{H}^1(\R^6)} \Vert v^2 \Vert_{\dot{H}^1(\R^6)} \Vert v^3 \Vert_{\dot{H}^1(\R^6)}.
\end{aligned}
\end{equation*}

To construct the exact kinetic energy, we use Arithmetic-Geometry Mean-Value inequality, i.e.,
\[
\frac1n \sum_{i=1}^n a_i \ge \left( \prod_{i=1}^n a_i \right)^{\frac1n}, \qquad \forall ~ a_i \ge 0, ~ i = 1, 2, \cdots, n.
\]
to get by assigning coefficients
\begin{equation*}
\begin{aligned}
\vert V(\b{\rm v}) \vert^2 \le & ~ \frac8{\kappa_1\kappa_2\kappa_3} \left( \frac{\kappa_1}2\Vert v^1 \Vert_{\dot{H}^1(\R^6)}^2 \right) \times \left( \frac{\kappa_2}2\Vert v^2 \Vert_{\dot{H}^1(\R^6)}^2 \right) \times \left( \frac{\kappa_3}2\Vert v^3 \Vert_{\dot{H}^1(\R^6)}^2 \right) \\
\le & ~ \frac8{\kappa_1\kappa_2\kappa_3} \left( \frac{\frac{\kappa_1}2\Vert v^1 \Vert_{\dot{H}^1(\R^6)}^2 + \frac{\kappa_2}2\Vert v^2 \Vert_{\dot{H}^1(\R^6)}^2 + \frac{\kappa_3}2\Vert v^3 \Vert_{\dot{H}^1(\R^6)}^2}3 \right)^3 \\
= & ~ \frac8{27\kappa_1\kappa_2\kappa_3} [K(\b{\rm v})]^3.
\end{aligned}
\end{equation*}
This completes the proof.
\end{proof}

\begin{remark}
For classical nonlinear Schr\"odinger equations, it is obvious that $ E(u) \lesssim K(u) $.
This is because
\begin{itemize}
  \item for focusing case, we can use the fact that $ V(u) \ge 0 $ to directly get
  \[
  E(u) = K(u) - V(u) \le K(u);
  \]
  \item for defocusing case, by Proposition \ref{GN}, we know
  \[
  E(u) = K(u) + V(u) \le K(u) + C_{GN} [K(u)]^{\frac32} \le [1 + C_{GN} [K(w)]^{\frac12} ] K(u).
  \]
\end{itemize}
It is disappointing that both methods above fail for \eqref{NLS system}, so we have to find a new way.
\end{remark}

\begin{proposition}[Energy Trapping]\label{Energy trapping}
	Let $ \b{\rm u} $ be a solution to \eqref{NLS system} with initial data $ \b{\rm u}_0 $ and maximal life-span $ I_{\max} \ni 0 $.
	If $ E(\b{\rm u}_0) \le (1 - \rho) E(\b{\rm W}) $ and $ K(\b{\rm u}_0) \le (1 - \rho^{\prime}) K(\b{\rm W}) $, then
	\begin{equation}
	K(\b{\rm u}(t)) \sim E(\b{\rm u}(t)), \qquad \forall ~ t \in I_{\max}.
	\label{EH}
	\end{equation}
\end{proposition}

\begin{proof}
On the one hand, By Lemma \ref{Ck} and \eqref{<}, we obtain the upper bound estimate for the energy $ E(\b{\rm u})$ as follow:
	\begin{equation*}
	\begin{aligned}
	E(\b{\rm u}(t)) \le & ~ K(\b{\rm u}(t)) + \vert V(\b{\rm u}(t)) \vert \\
	\le & ~ K(\b{\rm u}(t)) + C(\kappa_1, \kappa_2, \kappa_3) \left[ K(\b{\rm u}(t)) \right]^{\frac{3}{2}} \\
	\le & ~ \left( 1 + C(\kappa_1, \kappa_2, \kappa_3) \left[ (1 - \rho^\prime) K(\b{\rm W}) \right]^{\frac12} \right) K(\b{\rm u}(t)).
	\end{aligned}
	\end{equation*}
On the other hand, from the proof of Proposition \ref{coer}, we get the lower bound of $ E(\b{\rm u}) $ as following:
	\begin{equation*}
	\begin{aligned}
	E(\b{\rm u}(t))
	= & ~ \frac{1}{3} K(\b{\rm u}(t)) + \frac{1}{3} \left[ 2K(\b{\rm u}(t)) - 3V(\b{\rm u}(t)) \right] \\
	\ge & ~ \frac{1}{3} K(\b{\rm u}(t)) + \frac{1}{3}\rho^{\prime\prime} K(\b{\rm u}(t)) \\
	= & ~ \frac{1}{3} (1 + \rho^{\prime\prime}) K(\b{\rm u}(t)).
	\end{aligned}
	\end{equation*}
Combining the above two inequalities, we deduce the desired result\eqref{EH}.
\end{proof}

%散射
\section{Global Well-Posedness and Scattering}\label{GWPS}
In this section, we will prove the Theorem\ref{Sb} according to the contradiction argument that is concentration compactness/rigidity method.
If Theorem\ref{Sb} fails, we can find a very special type solution called critical solution which is mentioned in Theorem\ref{red}.
By analyzing the frequency scale functions, we clarify the critical solution into three types enemies which are introduced in Proposition \ref{enemies}.
More specifically, we use the Virial-type identity, mass conserved quantity and new physically conserved quantity \eqref{New} to exclude three types of enemies.
This fact shows that this critical solution does not exist essentially, which leads to the scattering theory.

We first give the definition of almost periodic solution and some significant conclusions.

\begin{definition}[Almost periodicity modulo symmetries]\label{almost periodic}
	 A solution $ \b{\rm u} = (u^1, u^2,u^3)^T $ to \eqref{NLS system} is called \emph{almost periodic modulo symmetries} with lifespan $ I \ni 0 $, if there exist functions $ \lambda : I \to \R^+ $, $ x : I \to \R^6 $, and $ C : \R^+ \to \R^+ $ such that
	\[
	\int_{\vert x - x(t) \vert \ge \frac{C(\eta)}{\lambda(t)}} \left(\frac{\kappa_1}{2} \vert \nabla u^1(t, x) \vert^2 + \frac{\kappa_2}{2} \vert \nabla u^2(t, x) \vert^2 +\frac{\kappa_3}{2} \vert \nabla u^3(t, x) \vert^2 \right)  {\rm d}x \le \eta
	\]
	and
	\[
	\int_{\vert \xi \vert \ge C(\eta)\lambda(t)} \left(\frac{\kappa_1}{2} \vert \xi \vert^2 \vert \hat{u^1}(t, \xi) \vert^2 + \frac{\kappa_2}{2} \vert \xi \vert^2 \vert \hat{u^2}(t, \xi) \vert^2 + \frac{\kappa_3}{2} \vert \xi \vert^2 \vert \hat{u^3}(t, \xi) \vert^2 \right) {\rm d}\xi \le \eta
	\]
	for all $ t \in I $ and $ \eta > 0 $.
	We refer to the function $ \lambda $ as the \emph{frequency scale function}, $ x $  as the \emph{spatial center function}, and $ C $ as the \emph{compactness modulus function}.
\end{definition}

By the Arzela--Ascoli Theorem, we can obtain some compactness results as follows:

\begin{lemma}[Compactness in $ L^2 $]\label{compactness in L^2}
A family of functions  denoted by $ \mathcal{F} \subset L^2(\R^d) $ is compact in $ L^2(\R^d) $ if and only if the following conditions hold:
	
(i) there exists $ A > 0 $ such that $ \Vert f \Vert_{L^2(\R^d)} \le A, \quad \forall ~ f \in \mathcal{F} $;
	
(ii) for any $ \varepsilon > 0 $, there exists $ R = R(\varepsilon) > 0 $ such that
\[
\int_{\vert x \vert \ge R} \vert f(x) \vert^2 {\rm d}x + \int_{\vert \xi \vert \ge R} \vert \hat{f}(\xi) \vert^2 {\rm d}\xi < \varepsilon, \qquad \forall ~ f \in \mathcal{F}.
\]
\end{lemma}

\begin{remark}\label{compactness in H^1}
By Lemma \ref{compactness in L^2}, a family of functions denoted by  $ \mathcal{F}_1 \subset {\dot H}^1(\R^d) $ is compact in $ {\dot H}^1(\R^d) $ if and only if $ \mathcal{F}_1 $ is bounded in $\dot{H}^1$ uniformly and for any $ \eta > 0 $,
there exists a compactness modulus function $ C(\eta) > 0 $ such that
\begin{equation*}
\int_{\vert x \vert \ge C(\eta)/\lambda(t)} \vert \nabla f(x) \vert^2 {\rm d}x + \int_{\vert \xi \vert \ge C(\eta)\lambda(t)} \vert \xi \vert^2 \vert \hat{f}(\xi) \vert^2 {\rm d}\xi < \eta, \qquad \forall ~ f \in \mathcal{F}_1.
\end{equation*}
\end{remark}

Assume that Theorem\ref{Sb} fails, we can prove the existence of critical solution which enjoys the property of almost periodicity modulo symmetries and has the minimal kinetic among all blow-up solutions.
\begin{theorem}[Reduction to almost periodic solutions] \label{red}
	Suppose Theorem \ref{Sb} fails. Then there exists a maximal-lifespan solution $\b{\rm u}_c : I_c \times \R^6 $ such that:
	\begin{align*}
		\sup\limits_{t\in I_c} K(\b{\rm u}_c) < K(\b{\rm W}),
	\end{align*}
$\b{\rm u}_c$ is almost periodic modulo symmetries, and $\b{\rm u}_c$ blows up both forward and backward in time. Moreover, $\b{\rm u}_c$ has minimal kinetic energy among all blow up solutions, that is
\begin{align*}
	\sup\limits_{t\in I} K(\b{\rm u}) \ge 	\sup\limits_{t\in I_c} K(\b{\rm u}_c)
\end{align*}
for all maximal-lifespan solutions $\b{\rm u}_c : I \times \R^6 \to \C^3$ that blows up at least one direction.
\end{theorem}

We will prove this theorem in subsection 4.1 by nonlinear profiles decomposition and stability lemma. With this in hand,
we can infer the critical solution can only be one of the following three cases via analyzing the frequency scale functions.

\begin{proposition}[Classification of critical solution]\label{enemies}
If Theorem \ref{Sb} fails, then there exists a minimal kinetic energy blow-up solution $ \b{\rm u}_c : I_c \times \R^6 \to \C^2 $ which is almost periodic modulo symmetries with the maximal interval of existence $ I_c $ satisfying
\begin{equation}
S_{I_c} (\b{\rm u_c}) = \infty, \quad \text{and} \quad
\sup_{t \in I_c} K(\b{\rm u}_c(t)) < K(\b{\rm W}).
\label{u_c}
\end{equation}
Furthermore, $ \b{\rm u}_c $ has to be one of the following three cases:
\begin{enumerate}
\item a \emph{blowing-up solution in finite time} if
\[
\vert \inf I_c \vert < \infty, \quad \text{or} \quad \sup I_c < \infty;
\]
\item a \emph{soliton-type solution} if
\[
I_c = \R, \quad \text{and} \quad \lambda(t) = 1, ~ \forall ~ t \in \R;
\]
\item a \emph{low-to-high frequency cascade} if
\[
I_c = \R, \quad \text{and} \quad \inf_{t \in \R} \lambda(t) = 1, \quad \limsup_{t \to +\infty} \lambda(t) = \infty.
\]
\end{enumerate}
\end{proposition}

And the critical solution satisfies the reduce Duhamel's formula, which is a useful tool in the proof of negative regularity in Subsection \ref{nr}.

\begin{lemma}[Reduced Duhamel's formula, \cite{Killip2013, Merle1998, Zhang2008}]\label{reduced}
Let $ \b{\rm u} $ be an almost periodic solution to \eqref{NLS system} with maximal-lifespan $ I $.
Then, for all $ t \in I $,
\begin{equation*}
\begin{aligned}
\b{\rm u}(t) & ~ = i \lim_{T \to \sup I} \int_t^T \verb"S"(t-s) \b{\rm f}(\b{\rm u}(s)) {\rm d}s \\
& ~ = -i \lim_{T \to \inf I} \int_T^t \verb"S"(t-s) \b{\rm f}(\b{\rm u}(s)) {\rm d}s,
\end{aligned}
\end{equation*}
as weak limits in $ {{\dot{\rm H}}^1(\R^6)} $.
\end{lemma}

%临界元的存在性
\subsection{Existence of critical solution}\label{eocs}
In this part, we prove Theorem \ref{red}.
For any $0\le K_0 \le K(\b{\rm W})$, we define
\[
L(K_0):= \sup \left\{ S_I (\b{\rm u}) ~ \Big\vert ~ \b{\rm u} : I \times \R^6 \to \C^2, ~ \sup_{t \in I} K(\b{\rm u}(t)) \le K_0 \right\}
\]
where the supremum is taken over all solutions $ \b{\rm u} : I \times \R^6 \to \C^3 $ to \eqref{NLS system} satisfying $ \sup_{t \in I} K(\b{\rm u}(t)) \le K_0 $. It is not difficult to check that $ L : [0, K(\b{\rm W})] \to [0, \infty] $ is non-decreasing and satisfies $ L(K(\b{\rm W})) = \infty $.
 On the other hand, by the global existence and scattering result of small data, Theorem \ref{sd} , we know
\[
L(K_0) \le C K_0^2, \quad \forall ~  K_0 < \delta_{sd},
\]
where $ \delta_{sd} > 0$ is a threshold in Theorem \ref{sd}. By stability Proposition \ref{pt}, we know $ L $ is a continuous function, which implies there exists a unique critical kinetic energy $ K_c $ such that
\begin{equation}
 L(K_0) \left\{
\begin{aligned}
< \infty, \quad K_0 < K_c, \\
= \infty, \quad K_0 \ge K_c.
\end{aligned}
\right.
\label{K_c}
\end{equation}

In particular, if $ \b{\rm u} : I \times \R^6 \to \C^3 $ is a maximal-lifespan solution to \eqref{NLS system} such that $ \sup_{t \in I} K(\b{\rm u}(t)) < K_c $, then $ \b{\rm u} $ is global and
\[
S_{\R}(\b{\rm u}) \le L \left( \sup_{t \in \R}K(\b{\rm u}(t)) \right) < \infty.
\]
From above argument, the failure of Theorem \ref{Sb} is equivalent to $ 0 < K_c < K(\b{\rm W}) $.

According to the profile decomposition Lemma \ref{lpd} and using the standard process, we can establish Palais-Smale condition modulo symmetries as follow.
\begin{proposition}[Palais-Smale condition modulo symmetries, \cite{Visan2012}]\label{PS}
	Let $ \b{\rm u}_n : I_n \times \R^6 \to \C^3$ be a sequence of solutions to \eqref{NLS system} such that
	\[
	\limsup_{n\to \infty} \sup_{t\in I_n} K(\b{\rm u}_n(t)) = K_c
	\]
	and $ \{ t_n \}_{n=1}^\infty \subset I_n $ be a sequence of time such that
	\begin{equation}
	\lim_{n\to \infty} S_{\ge t_n} (\b{\rm u}_n) = \lim_{n\to \infty} S_{\le t_n} (\b{\rm u}_n) = \infty.
	\label{S_n}
	\end{equation}
	Then there exists a subsequence of $ \{ \b{\rm u}_n(t_n) \}_{n=1}^\infty $  which converges in $ {\rm\dot H}^1_x(\R^6) $ modulo symmetries.
\end{proposition}

\begin{proof}[Proof of Theorem  \ref{red}]
Since  Theorem \ref{Sb} fails, so the critical kinetic energy must satisfy $K_c < K(\b{\rm W})$. By the definition of the critical kinetic energy, we can choose a sequence  $ \b{\rm u}_n : I_n \times \R^6 \to \C^3 $ of solutions to \eqref{NLS system} with $ I_n $ compact, such that
\begin{equation}
\sup_{n \in \N} \sup_{t \in I_n} K(\b{\rm u}_n (t)) = K_c, \quad \text{and} \quad \lim_{n \to \infty} S_{I_n}(\b{\rm u}_n) = \infty.
\label{u=K_c}
\end{equation}

Assuming that $t_n \in I_n$ satisfy $S_{\ge t_n}(\b{\rm u}_n)=S_{\le t_n}(\b{\rm u}_n)$, we have
\begin{equation}\label{qq}
	\lim_{n \to \infty} S_{\ge t_n}(\b{\rm u}_n) =\lim_{n \to \infty}  S_{\le t_n}(\b{\rm u}_n)= \infty
\end{equation}
Using translation invariance, we can assume $t_n=0$.
By  Proposition \ref{PS}, we can find a $\mathcal{T}_{\lambda_n} \in \mathcal{G}$, and $\b{\rm u}_{c, 0} \in {\rm\dot H}^1(\R^6)$ such that
\[
\lim_{n \to \infty} \Vert \mathcal{T}_{\lambda_n} \b{\rm u}_n (0) - \b{\rm u}_{c, 0} \Vert_{{\rm\dot H}^1(\R^6)} = 0.
\]
For convenience, we write for
\begin{equation} \label{str}
\lim_{n \to \infty} \Vert  \b{\rm u}_n (0) - \b{\rm u}_{c, 0} \Vert_{{\rm\dot H}^1(\R^6)} = 0.
\end{equation}

Let $ \b{\rm u}_c : I_c \times \R^6 \to \C^2 $ be the maximal-lifespan solution to \eqref{NLS system} with initial data $ \b{\rm u}_c(0) = \b{\rm u}_{c, 0} $.
By \eqref{str}, we can use the stability Proposition \ref{pt} to know $ I_c \subset \liminf I_n $ and
\[
\lim_{n \to \infty} \Vert \b{\rm u}_n - \b{\rm u}_c \Vert_{L_t^\infty(K, {\rm\dot H}_x^1(\R^6))} = 0, \quad \text{for all compact} ~ K \subset I_c.
\]
From \eqref{K_c}, we know

\begin{equation}
\sup_{t \in I} K(\b{\rm u}(t)) \le K_c.
\label{<K_c}
\end{equation}

Next, we show the $\b{\rm u}_c$ blows up both forward and backward in time. Indeed, if $\b{\rm u}_c$ does not blow up forward in time, we can obtain
\begin{equation}
	[0, \infty) \subset I_c, \quad S_{\ge 0} (\b{\rm u}_c) < \infty .
\end{equation}
By stability Proposition \ref{pt}, for sufficiently large $n$, we have
\begin{equation}
	S_{\ge 0} (\b{\rm u}_n) < \infty,
\end{equation}
which is a contradiction with \eqref{qq}.  By similar discussion, we know $\b{\rm u}_c$ also blows up backward in time.
Therefore, by definition of $K_c$,
\begin{equation}
\sup_{t \in I_c} K(\b{\rm u}_c(t)) \ge K_c.
\label{>K_c}
\end{equation}
Combining \eqref{<K_c} and \eqref{>K_c}, we can obtain
\begin{equation}
\sup_{t \in I_c} K(\b{\rm u}_c(t)) = K_c.
\label{=K_c}
\end{equation}

Finally, we prove $\b{\rm u}_c$ is almost periodic modulo symmetries. Considering a sequence $ \{ \tau_n \}_{n=1}^\infty \subset I_c $,
Since the fact that $ \b{\rm u}_c $ blows up in both time directions, we have
\[
S_{\ge \tau_n} (\b{\rm u}_c) = S_{\le \tau_n} (\b{\rm u}_c) = \infty.
\]
According to Proposition \ref{PS}, there exists a subsequence of $ \b{\rm u}_c (\tau_n) $, which is convergent in $ {\rm\dot H}_x^1 (\R^6) $ modulo symmetries.
This fact implies the desired result.
\end{proof}

\begin{remark}{\bf Classification of critical solution.}
We will use the similar argument in \cite{Killip2010} to prove  Proposition \ref{enemies}.

Firstly, we define some notations as follows:
\begin{equation}
	{\rm osc}(T) := \inf_{t_0 \in J} \frac{\sup \{ \lambda(t) ~ \vert ~ t \in J ~ \text{and} ~ \vert t - t_0 \vert \le T [\lambda(t_0)]^{-2} \}}{\inf \{ \lambda(t) ~ \vert ~ t \in J ~ \text{and} ~ \vert t - t_0 \vert \le T [\lambda(t_0)]^{-2} \}} \quad T>0,
\end{equation}
and
\begin{equation}
	 a(t_0) := \frac{\lambda(t_0)}{\sup \{ \lambda(t) ~ \vert ~ t \in J ~ \text{and} ~ t \le t_0 \}} + \frac{\lambda(t_0)}{\sup \{ \lambda(t) ~ \vert ~ t \in J ~ \text{and} ~ t \ge t_0 \}}  \quad t_0 \in J.
\end{equation}

Then, by standard progress, we can classify the minimal-kinetic-energy blow-up solution (critical solution) into three cases.

{\rm (1)}
\[
\lim_{T \to \infty} {\rm ost}(T) = \infty, ~ \inf_{t_0 \in J} a(t_0) = 0, ~ I_c \ne \R;
\] corresponds to the case of finite time blow-up,

{\rm (2)}
\[
\lim_{T \to \infty} {\rm ost}(T) < \infty;
\] corresponds to the case of soliton-type solution,

{\rm(3)}
\[
\lim_{T \to \infty} {\rm ost}(T) = \infty, ~ \inf_{t_0 \in J} a(t_0) = 0, ~ I_c = \R
\]
or
\[
\lim_{T \to \infty} osc(T) = \infty, ~ \inf_{t_0 \in J} a(t_0) > 0
\] corresponds to the case of low-to-high frequency cascade.
For more details, one can refer to \cite{Killip2010}.
\end{remark}

%爆破解排除
\subsection{Finite time blow-up}\label{FF}
In this part, we recall Virial identity and new physically conserved quantity and use them to exclude the finite time blow-up solution.
\begin{theorem}
There are no critical solution in the sense of  Proposition \ref{enemies} that is blow-up finite time.
\end{theorem}

\begin{proof}
We prove this theorem by contradiction argument. Suppose $ \b{\rm u}_c : I_c \times \R \to \C^3 $ is a finite time blow-up solution to \eqref{NLS system} with the maximal-lifespan $I_c$. Without loss of generality, let $\b{\rm u}_c$ blow up forward in time, i.e. $\sup I_c < \infty$.

We conclude that
\begin{equation}
\liminf_{t \to \sup I_c} \lambda(t) = \infty.
\label{lam=infty}
\end{equation}
In fact, if this claim fails, we can choose $t_n \in I_c$ satisfying $ t_n \to \sup I_c $, and define rescaled function $ \b{\rm v}_n : I_n \times \R^6 \to \C^3 $  as follow:
\begin{equation*}
\b{\rm v}_n (t, x) := \frac1{[\lambda(t_n)]^2} \b{\rm u}_c \left( t_n + \frac{t}{[\lambda(t_n)]^2}, x(t_n) + \frac{x}{\lambda(t_n)} \right),
\end{equation*}
where $ I_n := \left\{ t ~ \big\vert ~ t_n \in I_c, t_n + \frac{t}{[\lambda(t_n)]^2} \in I_c \right\} $.
Obviously, by the definition of almost periodicity solution, $ \{ \b{\rm v}_n (t, x) \}_{n=1}^\infty $ is a sequence of solutions to \eqref{NLS system} and $ \{ \b{\rm v}_n(0) \}_{n=1}^\infty \subset {\rm\dot H}^1(\R^6) $ is precompact.
Therefore, there exists $ \b{\rm v}_0 $ such that
\[
\lim_{n \to \infty} \Vert \b{\rm v}_n(0) - \b{\rm v}_0 \Vert_{{\rm\dot H}_x^1(\R^6)} = 0.
\]
According to $ K(\b{\rm v}_n(0)) = K(\b{\rm u}_c(t_n)) $ and Sobolev embedding, we know $ \b{\rm v}_0 \not\equiv \b{\rm 0} $.
we define  the maximal-lifespan solution $\b{\rm v}$ to

\begin{equation}
\left\{
\begin{aligned}
& i \partial_{t}\b{\rm v} + A \b{\rm v}  =  \b{\rm f}(\b{\rm v}), \quad (t, x) \in I \times \mathbb{R}^6,\\
& \b{\rm v}(0,x)= \b{\rm v}_0(x),
\end{aligned}
\right.
\label{ bijin}
\end{equation}
where $I:=\left(-T_-, T_+\right)$ satisfies $ -\infty \le -T_- < 0 < T_+ \le +\infty.$ By the local well-posedness and stability theory, we know when $n$ is large enough, $\b{\rm v}_n $ is local well-posed and for any compact subinterval $J \Subset \left(-T_-, T_+\right) $ the scattering size $ S_{J} (\b{\rm v}_n) < \infty $. This shows that $\b{\rm u}_c$ is well-posed and has finite scattering size in the interval $\left\{ t_n +\frac{t}{[\lambda(t_n)]^2}, \quad t\in J \right\}$.

On the other hand, when $t_n \to \sup I_c$, $\liminf_{t \to \sup I_c} \lambda(t) = \infty.$ From this, we know  $ \sup I_c \in \left\{ t_n +\frac{t}{[\lambda(t_n)]^2}, \quad t\in J \right\}$. This is a contradiction with forward blow-up on $I_c$.

Consider $ \b{\rm u}_c := (u^1_c, u^2_c, u^3_c)^T $.
In fact, let $ 0 < \eta < 1 $ and $ t \in I_c $. By H\"older's inequality and Sobolev embedding, we obtain
\begin{align}
\int_{\vert x \vert < R} \vert u^i_c (t, x) \vert^2 {\rm d}x
\le & ~ \int_{\vert x - x(t) \vert \le \eta R} \vert u^i_c(t, x) \vert^2 {\rm d}x + \int_{\vert x \vert \le R, \vert x - x(t) \vert > \eta R} \vert u^i_c(t, x) \vert^2 {\rm d}x \nonumber \\
\lesssim & ~ \eta^2 R^2 \Vert u^i_c \Vert_{L_x^3(\R^6)}^2 +  R^2 \left( \int_{\vert x - x(t) \vert > \eta R} \vert u_c(t, x) \vert^3 {\rm d}x \right)^{\frac23} \nonumber \\
\lesssim & ~ \eta^2 R^2  K(\b{\rm W}) +  R^2 \left( \int_{\vert x - x(t) \vert > \eta R} \vert u_c(t, x) \vert^3 {\rm d}x \right)^{\frac23}. \label{eta<1}
\end{align}
Letting $ \eta \to 0 $ and using the compactness of almost periodicity solution, we can set $ \zeta = \zeta(\varepsilon, R) = 2^{-\frac94} \varepsilon^{\frac32} R^{-3} > 0, \forall ~ \varepsilon > 0 $.
Thanks to \eqref{lam=infty}, there exists $ \eta = \eta(\varepsilon, R, \zeta) $ such that $ 0 < \eta \le 2^{-\frac34} [H(\b{\rm W})]^{-\frac12} \varepsilon^{\frac12} R^{-1} $ and $ C(\eta) \ge \lambda(t) C(\zeta) $.
Thus,
\[
\eta^2 R^2 \Vert u_c \Vert_{L_x^3(\R^6)}^2 \le \eta^2 R^2 \Vert u_c \Vert_{{\dot H}_x^1(\R^6)}^2 < \eta^2 R^2 K(\b{\rm W}) \le \frac{\varepsilon}{2}
\]
and
\begin{equation*}
\begin{aligned}
R^2 \left( \int_{\vert x - x(t) \vert > \eta R} \vert u_c(t, x) \vert^3 {\rm d}x \right)^{\frac23}
\le & ~ R^2 \left( \int_{\vert x - x(t) \vert > C(\zeta) \frac{\eta R \lambda(t)}{C(\eta)}} \vert \nabla u_c(t, x) \vert^2 {\rm d}x \right)^{\frac23} \\
< & ~ R^2 \zeta^{\frac23} = \frac{\varepsilon}{2}.
\end{aligned}
\end{equation*}
Combining the two above inequalities and \eqref{eta<1}, we have for $ i = 1, 2, 3 $
\begin{equation}
\limsup_{t \to \sup I_c} \int_{\vert x \vert < R} \vert u^i_c (t, x) \vert^2 {\rm d}x = 0, \quad \forall ~ R > 0.
\label{M=000}
\end{equation}

Let $ a(x) $ be a radial smooth function  satisfying
\[
a(x) = \left\{
\begin{aligned}
1, \qquad & \vert x \vert \le R, \\
0, \qquad & \vert x \vert \ge 2R,
\end{aligned}
\right.
\]
and $ \vert \nabla a(x) \vert \lesssim \frac{1}{\vert x \vert}$.
Let
\[
V_{R,c} (t) := \int_{\mathbb{R}^6} \left( \frac12 \vert u_c^1 \vert^2 + \frac12\vert u_c^2 \vert^2 + \vert u_c^3\vert^2 \right) a(x) {\rm d}x.
\]

On one hand, by \eqref{M=000}, we have
\begin{equation}
\limsup_{t \to \sup I_c} V_{R,c} (t) = 0.
\label{M=01}
\end{equation}
Then, by Proposition \ref{vi} and Hardy's inequality, we obtain
\begin{align*}
\vert V_{R,c}^\prime (t) \vert & ~ =  \left\vert\rm{Im} \int_{\R^6} (\kappa_{1}\overline{u}^1_c\nabla u^1_c + \kappa_{2}\overline{u}^2_c\nabla u^2_c + 2\kappa_{3}\overline{u}^3_c\nabla u^3_c ) \cdot \nabla a(x) {\rm d}x \right\vert \\
& ~ \lesssim\Vert \nabla \b{\rm u}_c \Vert_{{\rm L}^2(\R^6)} \left\Vert \frac{\b{\rm u}_c}{\vert x \vert} \right\Vert_{{\rm L}^2(\R^6)} \\
& ~ \lesssim[K(\b{\rm u}_c)]^2 \lesssim [K(\b{\rm W})]^2.
\end{align*}

On the other hand, Fundamental Theorem of Calculus tells us,
\[
V_{R,c}(t_1) \lesssim V_{R,c}(t_2) + \vert t_1 - t_2 \vert [K(\b{\rm W})]^2, \qquad \forall ~ t_1, t_2 \in I_c.
\]
Let $ t_2 \to \sup I_c $ and using \eqref{M=01}, we obtain
\[
V_{R,c}(t_1) \lesssim \vert \sup I_c - t_1 \vert [K(\b{\rm W})]^2, \qquad \forall ~ t_1 \in I_c.
\]
By the new  physically conserved quantity \eqref{New} and the above estimate, we have
\begin{equation*}
\begin{aligned}
M(\b{\rm u}_{c, 0})
= & ~ \lim_{\R \to \infty} V_{R,c}(t_1) \lesssim \vert \sup I_c - t_1 \vert [K(\b{\rm W})]^2, \qquad \forall ~ t_1 \in I_c.
\end{aligned}
\end{equation*}
Let $ t_1 \to \sup I_c $, we know $ \b{\rm u}_{c, 0} \equiv \b{\rm 0} $.

Thus, by the uniqueness of the solution to \eqref{NLS system}, $ \b{\rm u}_c \equiv \b{\rm 0} $, which is a contradiction with \eqref{u_c}.
\end{proof}

%负向正则性
\subsection{Negative regularity}\label{nr}

Now, we will show the negative regularity for the critical solution $\b{\rm u}_c$ to \eqref{NLS system}, which plays an important role in exclusion of soliton-type solutions and low-to-high frequency cascade.

Before we prove the negative regularity theorem, we need a technical tool. It is a form of Gronwall’s inequality that involves both the past and the future, “acausal” in the terminology of \cite{Colliander2006}.

\begin{lemma}[Bellman-Gronwall inequality, \cite{Killip2010}]\label{Gronwall}
Fix $ \gamma > 0 $. Given $ 0 < \eta < (1-2^{-\gamma})/2 $ and $ \{ b_k \}_{k=0}^{\infty} \in l^\infty(\mathbb{N}) $, let $ \{ x_k \}_{k=0}^{\infty} \in l^\infty(\mathbb{N}) $ be a non-negative sequence obeying
\begin{align}
x_k \leq b_k + \eta \sum_{l=0}^{\infty}2^{-\gamma|k-l|} x_l, \qquad \forall ~ k \ge 0.
\end{align}
Then
\begin{align}
x_k \lesssim \sum_{l=0}^{\infty} r^{|k-l|} b_l
\end{align}
for some $ r = r(\eta)\in(2^{-\gamma},1) $.
Moreover, $ r \downarrow 2^{-\gamma} $ as $ \eta \downarrow 0 $.
\end{lemma}

\begin{theorem}[Negative regularity]\label{negative regularity}
Let $ \b{\rm u}_c $ be as in Proposition \ref{enemies} with $ I_c = \R $.
If
\[
\sup_{t \in \R} K(\b{\rm u}_c(t)) < \infty \quad \text{and} \quad \inf_{t \in \R} \lambda(t) \ge 1,
\]
then there exists $ \eps > 0 $ such that $ \b{\rm u}_c \in {\rm L}_t^\infty(\R, {\rm\dot H}^{-\eps}(\R^6)) $.
\end{theorem}

\begin{proof}
Since $ \b{\rm u}_c(t) $ is  almost periodic modulo symmetries with $ \lambda(t) \geq 1 $, then there exists $ N_0 = N_0(\eta) $ such that
\begin{equation}
\left\Vert \nabla ({\rm P}_{\leq N_0} \b{\rm u}_c) \right\Vert_{{\rm L}_t^\infty(\R, {\rm L}_x^2(\R^6))} \leq \eta, \qquad \forall ~ \eta > 0,
\label{Small-compact}
\end{equation}
where operator $ {\rm P} := diag(P; P; P) $.
The proof is divided by four steps.
	
\noindent {\bf Step 1. Breaking the scaling in Lebesgue space}.
To this end, we define a quantity as follow
\begin{equation}
A(N) := N^{-\frac12} \left\Vert {\rm P}_{N} \b{\rm u}_c \right\Vert_{{\rm L}_t^\infty(\R, {\rm L}_x^4(\R^6))},
\label{A(N)}
\end{equation}
for $ N \le 10 N_0 $.
	
\noindent {\bf Step 2. A recurrence lemma for $ A(N) $}.

\begin{lemma}[Recurrence]\label{recurrence}
Let $ A(N) $ be defined as above, for any $ N \leq 10 N_0 $, we have
\begin{equation}
A(N) \lesssim_{\b{\rm u}_c} \left(\frac{N}{N_0}\right)^{\frac12}+\eta\sum_{\frac{N}{10}\leq N_1\leq N_0}\left(\frac{N}{N_1}\right)^{\frac12}A(N_1)+\eta\sum_{N_1<\frac{N}{10}}\left(\frac{N_1}{N}\right)^{\frac12} A(N_1).
\label{AN}
\end{equation}
\end{lemma}

\begin{proof}	
Fix $ N \le 10 N_0 $.
Using time-translation symmetry, we just need to prove
\begin{align}\nonumber
N^{-\frac12} \| \b{\rm u}_{c, 0} \|_{{\rm L}^{4}} \lesssim \left(\frac{N}{N_0}\right)^{\frac12} + & ~ \eta \sum_{\frac{N}{10} \le N_1 \le N_0} \left( \frac{N}{N_1} \right)^{\frac12} A(N_1) \\
+ & ~ \eta \sum_{N_1 < \frac{N}{10}} \left( \frac{N_1}{N} \right)^{\frac12} A(N_1).
\end{align}

By Bernstein estimates, Lemma \ref{disper}, and Lemma \ref{reduced}, we know
\begin{equation*}
\begin{aligned}
N^{-\frac12} \| P_N \b{\rm u}_{c, 0} \|_{{\rm L}^{4}(\R^6)}
\leq & ~ N^{-\frac12} \Big\| \int_{0}^{N^{-2}} \verb"S"(- \tau) {\rm P}_N \b{\rm f}(\b{\rm u}_c(\tau)) {\rm d}\tau \Big\|_{{\rm L}_x^4(\R^6)} \\
& ~ + N^{-\frac12} \int_{N^{-2}}^{\infty} \big\| \verb"S"(- \tau) {\rm P}_N\b{\rm f}(\b{\rm u}_c(\tau)) \big\|_{{\rm L}_x^4(\R^6)} {\rm d}\tau \\
\lesssim & ~ N^{\frac12} \| {\rm P}_N \b{\rm f}(\b{\rm u}_c) \|_{{\rm L}_t^\infty(\R, {\rm L}_x^\frac43(\R^6))}.
\end{aligned}
\end{equation*}	
We define some notations as follows:
\begin{equation}
\b{\rm f}(\b{\rm z}) =
\left(
\begin{aligned}
f^1(\b{\rm z})\\
f^2(\b{\rm z})\\
f^3(\b{\rm z})
\end{aligned}
\right)
:=
 \left(
\begin{aligned}
-\overline{z_2}z_3 \\
-\overline{z_1}z_3 \\
-z_1 z_2
\end{aligned}
\right), \qquad \text{for} ~ \b{\rm z} := \left(
\begin{aligned}
z^1 \\
z^2 \\
z^3
\end{aligned}
\right).
\label{fz}
\end{equation}
Directly computing, we can obtain that
\begin{equation*}
f^1(\b{\rm u}_c) - f^1({\rm P}_{> N_0} \b{\rm u}_c)= -(P_{\le N_0} u^3_c) (P_{\le N_0} \overline{u^2_c})- (P_{> N_0} u^3_c) (P_{\le N_0} \overline{u^2_c}) - (P_{\le N_0} u^3_c) (P_{> N_0} \overline{u^2_c}),
\end{equation*}
\begin{equation*}
f^2(\b{\rm u}_c) - f^2({\rm P}_{> N_0} \b{\rm u}_c)= -(P_{\le N_0} u^3_c) (P_{\le N_0} \overline{u^1_c})- (P_{> N_0} u^3_c) (P_{\le N_0} \overline{u^1_c}) - (P_{\le N_0} u^3_c) (P_{> N_0} \overline{u^1_c}),
\end{equation*}
and
\begin{equation*}
f^3(\b{\rm u}_c) - f^3({\rm P}_{> N_0} \b{\rm u}_c)= -(P_{\le N_0} u^2_c) (P_{\le N_0} {u^1_c})- (P_{> N_0} u^2_c) (P_{\le N_0}{u^1_c}) - (P_{\le N_0} u^2_c) (P_{> N_0} {u^1_c}).
\end{equation*}
According to the Fundamental Theorem of Calculus, we get
\begin{align*}
f^l(\b{\rm z}^1) - f^l(\b{\rm z}^2)
= & ~ (\b{\rm z}^1-\b{\rm z}^2) \int_{0}^{1} \partial_{\b{\rm z}} f^l(\b{\rm z}^1 + \theta (\b{\rm z}^1-\b{\rm z}^2)) {\rm d}\theta \\
& ~ + \overline{(\b{\rm z}^1-\b{\rm z}^2)} \int_{0}^{1} \partial_{\bar{\b{\rm z}}} f^l\b{\rm z}^1 + \theta (\b{\rm z}^1-\b{\rm z}^2)) {\rm d}\theta, \quad l=1, 2 , 3.
\end{align*}
And we can compute the partial derivatives for \eqref{fz},
\begin{equation*}
\partial_{\b{\rm z}} f^1(\b{\rm z}) = \left(
\begin{aligned}
0 ~~ \\
0 ~~ \\
-\bar{z}_2
\end{aligned}
\right), \quad \partial_{\bar{\b{\rm z}}} f^1(\b{\rm z}) = \left(
\begin{aligned}
0 ~~ \\
-z_3 \\
0 ~~
\end{aligned}
\right), \quad \partial_{\b{\rm z}} f^2(\b{\rm z}) = \left(
\begin{aligned}
0 ~~ \\
0 ~~ \\
-\bar{z}_1
\end{aligned}
\right),
\end{equation*}
and
\begin{equation*}
\partial_{\bar{\b{\rm z}}} f^2(\b{\rm z}) = \left(
\begin{aligned}
-z_3 \\
0 ~~ \\
0 ~~
\end{aligned}
\right), \quad \partial_{\b{\rm z}} f^3(\b{\rm z}) = \left(
\begin{aligned}
-z_2 \\
-z_1 \\
0 ~~
\end{aligned}
\right), \quad \partial_{\bar{\b{\rm z}}} f^3(\b{\rm z}) = \b{\rm 0}.
\end{equation*}

For $ \b{\rm u}_c := (u_c^1, u_c^2, u_c^3)^T $,
\[
\b{\rm f}(\b{\rm u}_c) =
\left(
\begin{aligned}
f^1(\b{\rm u}_c)\\
f^2(\b{\rm u}_c)\\
f^3(\b{\rm u}_c)
\end{aligned}
\right)
:= -
\left(
\begin{aligned}
-(P_{\le N_0} \overline{u^2_c})(P_{> N_0} u^3_c)  \\
-(P_{\le N_0} \overline{u^1_c})(P_{> N_0} u^3_c)  \\
-(P_{\le N_0} {u^1_c})(P_{> N_0} u^2_c)
\end{aligned}
\right)
-
\left(
\begin{aligned}
-(P_{> N_0} \overline{u^2_c})(P_{\le N_0} u^3_c)  \\
-(P_{> N_0} \overline{u^1_c})(P_{\le N_0} u^3_c)  \\
-(P_{> N_0} {u^1_c})(P_{\le N_0} u^2_c)
\end{aligned}
\right)
\]

\[
+ \b{\rm f} ({\rm P}_{> N_0}{\b{\rm u}_c} ) + \b{\rm f}({\rm P}_{\frac{N}{10} \le \cdot \le N_0} \b{\rm u}_c) + \left(
\begin{aligned}
- {\rm P}_{<\frac{N}{10}} \b{\rm u}_c^T \int_{0}^{1} \partial_{\b{\rm z}} f^1({\rm P}_{\frac{N}{10} \le \cdot \le N_0} \b{\rm u}_c + \theta {\rm P}_{<\frac{N}{10}} \b{\rm u}_c){\rm d}\theta   \\
- {\rm P}_{<\frac{N}{10}} \b{\rm u}_c^T \int_{0}^{1} \partial_{\b{\rm z}} f^2({\rm P}_{\frac{N}{10} \le \cdot \le N_0} \b{\rm u}_c + \theta {\rm P}_{<\frac{N}{10}} \b{\rm u}_c){\rm d}\theta   \\
- {\rm P}_{<\frac{N}{10}} \b{\rm u}_c^T \int_{0}^{1} \partial_{\b{\rm z}} f^3({\rm P}_{\frac{N}{10} \le \cdot \le N_0} \b{\rm u}_c + \theta {\rm P}_{<\frac{N}{10}} \b{\rm u}_c){\rm d}\theta
\end{aligned}
\right)
\]

\[
+ \left(
\begin{aligned}
- \overline{{\rm P}_{<\frac{N}{10}} \b{\rm u}_c}^T \int_{0}^{1} \partial_{\bar{\b{\rm z}}} f^1({\rm P}_{\frac{N}{10} \le \cdot \le N_0} \b{\rm u}_c + \theta {\rm P}_{<\frac{N}{10}} \b{\rm u}_c){\rm d}\theta   \\
- \overline{{\rm P}_{<\frac{N}{10}} \b{\rm u}_c}^T \int_{0}^{1} \partial_{\bar{\b{\rm z}}} f^2({\rm P}_{\frac{N}{10} \le \cdot \le N_0} \b{\rm u}_c + \theta {\rm P}_{<\frac{N}{10}} \b{\rm u}_c){\rm d}\theta   \\
- \overline{{\rm P}_{<\frac{N}{10}} \b{\rm u}_c}^T \int_{0}^{1} \partial_{\bar{\b{\rm z}}} f^3({\rm P}_{\frac{N}{10} \le \cdot \le N_0} \b{\rm u}_c + \theta {\rm P}_{<\frac{N}{10}} \b{\rm u}_c){\rm d}\theta
\end{aligned}
\right)
:=  \mathcal{A}_1 +\mathcal{A}_2 + \mathcal{B + C} + \mathcal{D}_1 + \mathcal{D}_2.
\]

Firstly, we can directly estimate the first three terms above by H\"older and Bernstein inequalities,
\begin{equation*}
\left\Vert P_N (\mathcal{A}_1 +\mathcal{A}_2+ \mathcal{B}) \right\Vert_{{\rm L}_t^\infty(\R, {\rm L}_x^{\frac43}(\R^6))}
\lesssim \| P_{>N_0} \b{\rm u}_c \|_{{\rm L}_t^\infty(\R, {\rm L}_x^{\frac{12}5}(\R^6))} \| \b{\rm u}_c \|_{{\rm L}_t^\infty(\R, {\rm L}_x^{3}(\R^6))}
\lesssim N_0^{-\frac12}.
\end{equation*}

Next, we consider the $\mathcal{D}_1 + \mathcal{D}_2 $ term.
We only estimate the $\mathcal{D}_1$ term since the $\mathcal{D}_2$ one can be obtained similarly.
Bernstein's inequality yields
\begin{align*}
\| P_{>\frac{N}{10}}\partial_{\b{\rm z}} f^i(\b{\rm u}_c) \|_{{\rm L}_t^\infty(\R, {\rm L}_x^2(\R^6))}\lesssim N^{-1}\|\nabla\b{\rm u}_c \|_{{\rm L}_t^\infty(\R, {\rm L}_x^2(\R^6))}, \qquad i = 1, 2, 3.
\end{align*}

Thus, from H\"older's inequality and \eqref{Small-compact},
\begin{align*}
& ~ \left\Vert P_N \mathcal{D}_1 \right\Vert_{{\rm L}_t^\infty(\R, {\rm L}_x^{\frac43}(\R^6))} \\
= & ~ \sum_{i=1}^3\Big\| P_N \Big( P_{<\frac{N}{10}} \b{\rm u}_c \int_{0}^{1} \partial_{\b{\rm z}} f^i(P_{\frac{N}{10} \le \cdot \le N_0} \b{\rm u}_c + \theta P_{<\frac{N}{10}} \b{\rm u}_c) {\rm d}\theta \Big) \Big\|_{{\rm L}_t^\infty(\R, {\rm L}_x^{\frac43}(\R^6))} \\
\lesssim & ~ \| P_{< \frac{N}{10}} \b{\rm u}_c \|_{{\rm L}_t^\infty(\R, {\rm L}_x^4(\R^6))}
\sum_{i=1}^3 \Big\| P_{> \frac{N}{10}} \Big( \int_{0}^{1} \partial_{\b{\rm z}} f^i(P_{\frac{N}{10} \le \cdot \le N_0} \b{\rm u}_c + \theta P_{<\frac{N}{10}} \b{\rm u}_c){\rm d}\theta \Big) \Big\|_{{\rm L}_t^\infty(\R, {\rm L}_x^2(\R^6))} \\
\lesssim & ~ N^{-1} \| P_{< \frac{N}{10}} \b{\rm u}_c \|_{{\rm L}_t^\infty(\R, {\rm L}_x^4(\R^6))} \| \nabla P_{\le N_0} \b{\rm u}_c \|_{{\rm L}_t^\infty(\R, {\rm L}_x^2(\R^6))} \\
\lesssim & ~ N^{-1} \eta \sum_{N_1 < \frac{N}{10}} N_1^{\frac12} A(N_1).
\end{align*}

Therefore, we get
\begin{equation*}
\| P_{N}\mathcal{D} \|_{{\rm L}_t^\infty(\R, {\rm L}_x^{\frac43}(\R^6))}
= \sum_{j=1}^2 \| P_{N}\mathcal{D}_j \|_{{\rm L}_t^\infty(\R, {\rm L}_x^{\frac43}(\R^6))}
\lesssim N^{-1} \eta \sum_{N_1<\frac{N}{10}} N_1^{\frac12} A(N_1).
\end{equation*}

Now, only $\mathcal{C}$ term left, so we consider it by Bernstein and \eqref{Small-compact},
\begin{align*}
& ~ \left\Vert P_N\mathcal{C} \right\Vert_{{\rm L}_t^\infty(\R, {\rm L}_x^{\frac43}(\R^6))}
\lesssim \left\Vert \b{\rm f} (P_{\frac{N}{10} \le \cdot \le N_0} \b{\rm u}_c) \right\Vert_{{\rm L}_t^\infty(\R, {\rm L}_x^{\frac43}(\R^6))} \\
\lesssim & ~ \sum_{\frac{N}{10} \le N_1, N_2 \le N_0} \left( \left\Vert \overline{(P_{N_1} u^2_c)} (P_{N_2} u^3_c) \right\Vert_{L_t^\infty(\R, L_x^{\frac43}(\R^6))} + \left\Vert \overline{(P_{N_1} u^1_c)} (P_{N_2} u^3_c) \right\Vert_{L_t^\infty(\R, L_x^{\frac43}(\R^6))} \right. \\
& ~ \qquad \left. + \left\Vert (P_{N_1} u^1_c) (P_{N_2} u^2_c) \right\Vert_{L_t^\infty(\R, L_x^{\frac43}(\R^6))} \right) \\
\lesssim & ~ \eta \sum_{\frac{N}{10} \le N_1 \le N_2 \le N_0} N_2^{-1} \left( \left\Vert P_{N_1} u^1_c \right\Vert_{L_t^\infty(\R, L_x^4(\R^6))} + N_2^{-1} \left\Vert P_{N_1} u^2_c \right\Vert_{L_t^\infty(\R, L_x^4(\R^6))} \right. \\
& ~ \qquad \left. +  N_2^{-1} \left\Vert P_{N_1} u^3_c \right\Vert_{L_t^\infty(\R, L_x^4(\R^6))} \right) \\
& ~ + \sum_{\frac{N}{10} \le N_2 \le N_1 \le N_0} \left( \left\Vert P_{N_1} u^1_c \right\Vert_{L_t^\infty(\R, L_x^2(\R^6))} \left\Vert P_{N_2} u^2_c \right\Vert_{L_t^\infty(\R, L_x^4(\R^6))} \right. \\
& ~ \qquad + \left\Vert P_{N_1} u^1_c \right\Vert_{L_t^\infty(\R, L_x^2(\R^6))} \left\Vert P_{N_2} u^3_c \right\Vert_{L_t^\infty(\R, L_x^4(\R^6))} \\
& ~ \qquad \left. + \left\Vert P_{N_1} u^2_c \right\Vert_{L_t^\infty(\R, L_x^2(\R^6))} \left\Vert P_{N_2} u^3_c \right\Vert_{L_t^\infty(\R, L_x^4(\R^6))} \right) \\
\lesssim & ~ \eta \sum_{\frac{N}{10} \le N_1 \le N_0} N_1^{-\frac12} A(N_1) + \eta \sum_{\frac{N}{10} \le N_2 \le N_1 \le N_0} \Big( \frac{N_2}{N_1} \Big) (N_2^{-\frac12} A(N_2)) \\
\lesssim & ~ \eta \sum_{\frac{N}{10} \le N_1 \le N_0} N_1^{-\frac12} A(N_1).
\end{align*}

Finally, we combine the above estimates to get
\begin{equation*}
\begin{aligned}
A(N) & ~ = N^{\frac12} \| P_{N}(\mathcal{A}_1 +\mathcal{A}_2+  \mathcal{B + C} +\mathcal {D}_1 + \mathcal{D}_2 ) \|_{{\rm L}_t^\infty(\R, {\rm L}_x^{\frac43}(\R^6))} \\
& ~ \lesssim_{\b{\rm u}_c} \left(\frac{N}{N_0}\right)^{\frac12}+\eta\sum_{\frac{N}{10}\leq N_1\leq N_0}\left(\frac{N}{N_1}\right)^{\frac12}A(N_1)+\eta\sum_{N_1<\frac{N}{10}}\left(\frac{N}{N_1}\right)^{\frac12}A(N_1),
\end{aligned}
\end{equation*}
which yields \eqref{AN}.
\end{proof}

Combining the lemma \ref{Gronwall} and Lemma \ref{recurrence}, we can obtain the following the property.

\noindent {\bf Step 3. $ L^p $-breach of scaling}.

\begin{proposition}\label{re}
Let $ \b{\rm u}_c $ satisfy the all assumptions of Theorem \ref{negative regularity}. Then
\[
\b{\rm u}_c \in {\rm L}_t^\infty(\R, {\rm L}_x^p(\R^6)) \quad \text{for} ~ \frac{14}{5} \le p < 3.
\]
Especially, according to H\"older's inequality,
\[
\nabla \b{\rm f}(\b{\rm u}_c) \in {\rm L}_t^\infty(\R, {\rm L}_x^r(\R^6)) \quad \text{for} ~ \frac76 \le r < \frac65.
\]
\end{proposition}

\begin{proof}
From Lemma \ref{Gronwall} and Lemma \ref{recurrence}, we claim
\begin{equation}
\Vert P_N \b{\rm u}_c \Vert_{{\rm L}_t^\infty(\R, {\rm L}_x^{\frac43}(\R^6))} \lesssim N \quad \text{for} ~ N \leq 10 N_0.
\label{<N}
\end{equation}
In fact, taking $N=10\cdot2^{-k}N_0,$ $x_k=A(10\cdot2^{-k}N_0)$ and $\gamma=1$ in Lemma \ref{Gronwall}, we deduce
\[
A(10\cdot2^{-k}N_0)\lesssim \sum_{l=0}^{\infty} r^{-|k-l|}(10\cdot2^{-l})^\frac12.
\]
Then taking $ \eta > 0 $ small enough such that $ r < \sqrt{2} $, we have
\[
A(10\cdot2^{-k}N_0)\lesssim (10\cdot2^{-k})^\frac12\sum_{l=0}^{\infty} r^{|k-l|}(10\cdot2^{-l+k})^\frac12\lesssim (10\cdot2^{-k})^\frac12.
\]
Recall the definition of $ A(N) $ in \eqref{A(N)}, we get
\[
\Vert P_N \b{\rm u}_c \Vert_{{\rm L}_t^\infty(\R, {\rm L}_x^{\frac43}(\R^6))} \lesssim N_0^{-\frac12} N^{\frac12}(10\cdot2^{-k}N_0)^\frac12 \lesssim N,
\]
this yields \eqref{<N}.

Using interpolation inequality, we can obtain
\begin{equation*}
\| P_N \b{\rm u}_c \|_{{\rm L}_t^\infty(\R, {\rm L}_x^p(\R^6))}
\lesssim \| P_N \b{\rm u}_c \|_{{\rm L}_t^\infty(\R, L_x^4(\R^6))}^{\frac{2(p-2)}{p}} \| P_N \b{\rm u}_c \|_{{\rm L}_t^\infty(\R, {\rm L}_x^2(\R^6))}^{\frac4p-2}
\lesssim N^{3-\frac{8}{p}}
\lesssim N^{\frac17}
\end{equation*}
for all $ N \le 10 N_0 $.
According to Lemma \ref{Bern},
\begin{equation*}
\begin{aligned}
\| \b{\rm u}_c \|_{{\rm L}_t^\infty(\R, {\rm L}_x^p(\R^6))}
\le & ~ \| P_{\leq N_0} \b{\rm u}_c \|_{{\rm L}_t^\infty(\R, {\rm L}_x^p(\R^6))} + \| P_{>N_0} \b{\rm u}_c \|_{{\rm L}_t^\infty(\R, {\rm L}_x^p(\R^6))} \\
\lesssim & ~ \sum_{N \leq N_0} N^{\frac17} + \sum_{N > N_0} N^{2-\frac6p}
\lesssim 1,
\end{aligned}
\end{equation*}
which completes the proof of this Proposition.
\end{proof}	

And we can prove some negative regularity as shown in the following proposition by the technique of decomposition of frequency as in \cite{Killip2010}.

\begin{proposition}\label{sre}
	Suppose that $ \b{\rm u}_c $  satisfies all  assumptions of  Theorem \ref{negative regularity}.
	If $ |\nabla|^s \b{\rm f}(\b{\rm u}_c) \in {\rm L}_t^\infty(\R, {\rm L}_x^p(\R^6)) $, $ \frac76 \leq p < \frac65 $ and  $ s \in [0, 1] $,
	then there exists $ s_0 = s_0(p) > 0 $ such that $ \b{\rm u}_c \in {\rm L}_t^\infty(\R, {\rm\dot H}_x^{s-s_0+}(\R^6)) $.
\end{proposition}

\noindent {\bf Step 4. The proof of Theorem \ref{negative regularity}}.
First, we take $s=1$ in the Proposition \ref{sre}.
From this, we deduce $ \b{\rm u} \in {\rm L}_t^\infty({\rm\dot H}_x^{1-s_0+}(\R^6))$ for some $s_0 > 0 $.
Next, we use the fractional chain rule  Lemma \ref{cr} and  Proposition \ref{re} to obtain $ |\nabla|^{1-s_0+} \b{\rm f}(\b{\rm u}_c) \in {\rm L}_t^\infty(\R, {\rm L}_x^p(\R^6)) $ for some $ \frac76 \leq p < \frac65 $.
Using Proposition \ref{sre} again with $s=1-s_0^+$, we get $ \b{\rm u} \in {\rm L}_t^\infty({\rm\dot H}_x^{1-2s_0+}(\R^6))$.
Finally, we can iterate this procedure finite times to get the desired result.
\end{proof}	

%孤子解的排除
\subsection{Soliton-type}\label{SS}
We use the new physically conserved quantity \eqref{New} under mass-resonance by a clever observation to prove the  zero momentum of the critical solution, which plays a key role to exclude the soliton-type solution. Based on the zero momentum, we get some compactness conditions.
Combining this, we exclude the soliton-type solution.

\begin{proposition}\label{P=0}
Let $ \b{\rm u}_c := (u^1_c, u^2_c, u^3_c)^T $ be a minimal-kinetic-energy blowing-up solution to \eqref{NLS system} satisfying $ \b{\rm u}_c \in {\rm L}_t^\infty(I_c, {\rm H}_x^1(\R^6)) $ and  the mass-resonance condition, then the momentum
is zero, that is,
\begin{equation}
P(\b{\rm u}_c) := \Im \int_{\mathbb{R}^6} \left(  \overline{u}^1_{c} \nabla u^1_{c} + \overline{u}^2_{c}\nabla u^2_{c}  + \overline{u}^3_{c}\nabla u^3_{c}\right) {\rm d}x=\b{\rm 0}.
\label{p=0}
\end{equation}
\end{proposition}

\begin{proof}
Under the mass-resonance condition,  \eqref{NLS system} enjoys the Galilean invariance property:
\[
\b{\rm u}_c(t, x)
= \left(
\begin{aligned}
u^1_c(t, x) \\
u^2_c(t, x) \\
u^3_c(t,x)
\end{aligned}
\right) \to \left(
\begin{aligned}
u_c^{1,\xi}(t, x) \\
u_c^{2,\xi}(t, x) \\
u_c^{3,\xi}(t,x)
\end{aligned}
\right)
=: \left(
\begin{aligned}
e^{i\frac{x\cdot\xi}{\kappa_1}}e^{-it\frac{\vert\xi\vert^2}{\kappa_1}}u^1(t, x-2t\xi) \\
e^{i\frac{x\cdot\xi}{\kappa_2}}e^{-it\frac{\vert\xi\vert^2}{\kappa_2}}u^2(t, x-2t\xi) \\
e^{i\frac{x\cdot\xi}{\kappa_3}}e^{-it\frac{\vert\xi\vert^2}{\kappa_3}}u^3(t, x-2t\xi)
\end{aligned}
\right)
\]
for any $ \xi \in \mathbb{R}^6 $.

Then, for $ \b{\rm u}_c^{\xi} := (u_c^{1,\xi}, u_c^{2,\xi}, u_c^{3,\xi})^T $, directly computing, we have
\[
\begin{aligned}
K(\b{\rm u}_c^{\xi}) = & ~ \sum_{l=1}^3 \frac{\kappa_l}{2}\Vert \nabla u_c^{l,\xi} \Vert_{L^2(\R^6)}^2 = \sum_{l=1}^3 \frac{\kappa_l}{2} \left\Vert i \frac{\xi}{\kappa_l} u^l_c + \nabla u^l_c \right\Vert_{L^2(\R^6)}^2 \\
= & ~ \sum_{l=1}^3 \frac{\kappa_l}{2} \int_{\R^6} \left\vert \Re(\nabla u^l_c) - \frac{\xi}{\kappa_l} \Im u^l_c \right\vert^2 + \left\vert \frac{\xi}{\kappa_l} \Re u^l_c + \Im(\nabla u^l_c) \right\vert^2 {\rm d}x \\
= & ~ \sum_{l=1}^3 \frac{\kappa_l}{2}\int_{\R^6} \left\vert \nabla u^l_c \right\vert^2 + \frac {\vert \xi \vert^2}{\vert \kappa_l \vert ^2} \vert u^l_c \vert^2 + \frac{2\xi}{\kappa_l} \left( \Re u^l_c \Im(\nabla u^l_c) - \Re(\nabla u^l_c) \Im u^l_c \right) {\rm d}x.
\end{aligned}
\]
Since the \eqref{NLS system} satisfies the mass-resonance, we use the new conserved quantity \eqref{New} to obtain
\[
\begin{aligned}
& ~ K(\b{\rm u}_c^{\xi}) - K(\b{\rm u}_c) - \vert \xi \vert^2 \Lambda(\b{\rm u}_c) \\
= & ~ \xi \cdot \int_{\R^6} \big( \Re u^1_c \Im(\nabla u^1_c) - \Re(\nabla u^1_c) \Im u_c \big) \\
& ~ + \big( \Re u^2_c \Im(\nabla u^2_c) - \Re(\nabla u^2_c) \Im u^2_c \big) \\
& ~ + \big( \Re u^3_c \Im(\nabla u^3_c) - \Re(\nabla u^3_c) \Im u^3_c \big) {\rm d}x \\
= & ~ \xi \cdot \int_{\R^6} \Im(\overline{u^1_c} \nabla u^1_c) +  \Im(\overline{u^2_c} \nabla u^2_c) + \Im(\overline{u^3_c} \nabla u^3_c) {\rm d}x
= \xi \cdot P(\b{\rm u}_c).
\end{aligned}
\]
	
It is easy to see $ S_{I_c^{\xi}}(\b{\rm u}_c^{\xi}) = S_{I_c}(\b{\rm u}_c) = \infty $.
Therefore,  $ \b{\rm u}_c^{\xi} $ is also a blow-up solution to \eqref{NLS system} .
Recalling the fact that $ \b{\rm u}_c $  has the minimal kinetic energy among all the blowing-up solutions, we have
\[
\vert \xi \vert^2 \Lambda(\b{\rm u}_c) + \xi \cdot P(\b{\rm u}_c) = K(\b{\rm u}_c^{\xi}) - K(\b{\rm u}_c) \ge 0
\]
holds for any $ \xi \in \R^6 $.
Taking $ \xi = -\frac{P(\b{\rm u}_c)}{2\Lambda(\b{\rm u}_c)} $, we get
\[
0 \le \vert \xi \vert^2 \Lambda(\b{\rm u}_c) + \xi \cdot P(\b{\rm u}_c) = -\frac{\vert P(\b{\rm u}_c) \vert^2}{4\Lambda(\b{\rm u}_c)} \le 0,
\]
which implies \eqref{p=0}.
\end{proof}

Making the use of properties of negative regularity (Theorem \ref{nr}) and zero momentum (Proposition \ref{P=0}), we can control the speed of spatial center function of soliton-type solutions very well.

\begin{lemma}[Control over $ x(t) $] \label{x(t)}
Assuming that $ \b{\rm u}_c $ is a soliton-type solution as in Proposition \ref{enemies}.
Then for any $ \eta > 0 $, there exists $ C(\eta) > 0 $ such that
\[
\sup_{t \in \R} \int_{\vert x - x(t) \vert \ge C(\eta)} \vert \b{\rm u}_c(t, x) \vert^2 {\rm d}x \lesssim \eta.
\]
Furthermore, the spatial center function $ x(t) $ moves slower and slower, i.e.,
\[
\vert x(t) \vert = o(t), \qquad t \to \infty.
\]
\end{lemma}

Both Lemma \ref{x(t)} and the finite mass of soliton-type solution, i.e., $ \lambda(t) \ge 1 $, can help us use Theorem \ref{vi} more efficiently.

\begin{theorem}\label{Soliton-th}
If the \eqref{NLS system} satisfies the mass-resonance, then there does not exist the soliton-type solution among all minimal-kinetic-energy blow up solution to \eqref{NLS system}.
\end{theorem}

\begin{proof}
We prove the fact by contradiction.  Suppose $ \b{\rm u}_c : \R \times \R^6 \to \C^2 $ is a minimal-kinetic-energy blow up solution of soliton-type  to \eqref{NLS system}.
By Definition \ref{almost periodic} and $  {\dot H}^1(\R^6)\hookrightarrow L^3(\R^6) $, for any $ \eta > 0 $, there exists $ C(\eta) > 0 $ such that
\begin{equation}
\sup_{t \in \R} \int_{\vert x - x(t) \vert \ge C(\eta)} \left( \vert \nabla \b{\rm u}_c \vert^2 + \vert \b{\rm u}_c \vert^3 \right) {\rm d}x \le \eta.
\label{eta}
\end{equation}
Using Lemma \ref{x(t)}, we know there exists $ T_0 = T_0(\eta) \in \R $ such that
\begin{equation}
\vert x(t) \vert \le \eta t, \qquad \forall ~ t \ge T_0.
\label{x=0}
\end{equation}
	
Recall Theorem \ref{vi}, under the mass resonance condition \eqref{mr}, we take $ a(x) $ be a radial smooth function satisfying
\[
a(x) = \left\{
\begin{aligned}
\vert x \vert^2, \quad & \vert x \vert \le R, \\
0, \qquad & \vert x \vert \ge 2R,
\end{aligned}
\right.
\]
where $ R > 0 $ will be chosen later, we define
\[
V_R(t) :=2 \Im \int_{\mathbb{R}^6} \left(  \overline{u^1_c} \nabla u^1_c + \overline{u^2_c} \nabla u^2_c + \overline{u^3_c} \nabla u^3_c\right ) \cdot \nabla a(x) {\rm d}x
\]
By Theorem \ref{negative regularity}, we can get $ \b{\rm u}_c := (u^1_c, u^2_c, u^3_c)^T \in {\rm L}_t^\infty(\R, {\rm L}_x^2(\R^6)) $.
Then we get
\begin{equation}
|V_{R}(t)| \lesssim R K(\b{\rm u}_c) M(\b{\rm u}_c) \lesssim R
\label{<R}
\end{equation}
and
\begin{align}
V_R^{\prime}(t) = & ~ \Re \int_{\mathbb{R}^6} \left( 4\kappa_1\partial_{j}\overline{u^1} _{c}\partial_{k}u^1_{c} + 4\kappa_2\partial_{j} \overline{u^2}_{c} \partial_{k}u^2_{c}+ 4\kappa_3\partial_{j}\overline{u^3}_{c}\partial_{k}u^3_c \right) \partial_{jk} a(x) {\rm d}x \nonumber \\
& ~ - \int_{\mathbb{R}^6} \left(\kappa_1 \vert u^1_{c} \vert^2 + \kappa_2 \vert u^2_{c} \vert^2 +\kappa_3 \vert u^3_{c}\vert^2 \right) \Delta \Delta a(x) {\rm d}x \nonumber \\
& ~ - 2  \Re \int_{\mathbb{R}^6} \overline{u^1_{c}u^2_{c}} u^3_{c} \Delta a(x) {\rm d}x \nonumber \\
= & ~ 8 [ 2 K(\b{\rm u}_c) - 3 V(\b{\rm u}_c) ] + O\left( \int_{\vert x \vert \ge R} \vert \nabla \b{\rm u}_c \vert^2 + \vert \b{\rm u}_c \vert^3 {\rm d}x + \Vert \b{\rm u}_c \Vert_{{\rm L}^3(R \le \vert x \vert \le 2R)}^2 \right).  \nonumber	
\end{align}
Observing the coercivity of energy (Proposition \ref{coer}), energy trapping (Corollary \ref{Energy trapping}), and \eqref{eta}, we can choose $ \eta > 0 $ sufficiently small and take
\[
R := C(\eta) + \sup_{T_0 \le t \le T_1} \vert x(t) \vert
\]
for any $ T_0 < T_1 $. From \eqref{eta}, we have
\begin{equation}
V_{R}^{\prime\prime}(t) \gtrsim_{\kappa_1\kappa_2\kappa_3} E(\b{\rm u}_{c, 0}).
\label{Vpp}
\end{equation}
	
Applying the Fundamental Theorem of Calculus on $ [T_0, T_1] $, by \eqref{x=0}, \eqref{<R} and \eqref{Vpp}, we get
\[
(T_1 - T_0) E(\b{\rm u}_{c, 0})  \lesssim_{\kappa_1 \kappa_2 \kappa_3} R = C(\eta) + \sup_{T_0 \le t \le T_1} \vert x(t) \vert \le C(\eta) + \eta T_1, \quad \forall ~ T_1 > T_0.
\]
Setting first $ \eta \to 0 $ and then $ T_1 \to \infty $, we find $ E(\b{\rm u}_{c, 0}) = 0 $.
Using the conservation of energy and energy trapping again, we get $ K(\b{\rm u}_c(t)) = V(\b{\rm u}_c(t)) = 0, ~ \forall ~ t \in \R $. This implies $\b{\rm u}_c \equiv 0$, which is contradiction to the scattering size $S_{\mathbb{R}}(\b{\rm u}_c)=\infty$.
\end{proof}

\begin{remark}
Under the assumption that $ \b{\rm u}_0 $ is radial, we use the properties of radial function and odd function to write the momentum as follows:
\begin{equation*}
P(\b{\rm u}_c) := \Im \int_{\mathbb{R}^6} \left(  \overline{u^1(r)}_{c}u^{1 \prime}_{c}(r)\frac{x}{r} + \overline{u^2(r)}_{c} u^{2 \prime}_{c}(r)\frac{x}{r}  + \overline{u^3(r)}_{c} u^{3 \prime}_{c}(r) \frac{x}{r}\right) {\rm d}x = \b{\rm 0}.
\end{equation*}
Then we can set $ x(t) \equiv 0 $.
Let $ R := C(\eta) $, thus Theorem \ref{Soliton-th} also holds without mass-resonance.
\end{remark}

\subsection{Low-to-high frequency cascade}\label{LL}
In this part, we will use the negative regularity in Theorem \ref{negative regularity} and some compactness conditions of critical solutions to exclude the low-to-high frequency cascade.

\begin{theorem}\label{death3}
There are no global solutions to \eqref{NLS system} that are low-to-high frequency cascades as in the theorem \ref{enemies}.
\end{theorem}

\begin{proof}
We argue by contradiction and suppose $ \b{\rm u}_c : I_c \times \R \to \C^2 $ is a low-to-high frequency cascade solution to \eqref{NLS system}.
By Theorem \ref{negative regularity}, we know $ \b{\rm u}_c \in {\rm L}_t^\infty(\R, {\rm L}_x^2(\R^6)) $.
Noticing the conservation of mass, we obtain
\[
0 \le M(\b{\rm u}_c) = M(\b{\rm u}_c(t)) := \frac{1}{2}\Vert u_c^1 \Vert_{L^2}^2 + \frac{1}{2} \Vert u_c^2 \Vert_{L^2}^2 +\Vert u_c^3 \Vert_{L^2}^2< \infty,
\]
for any $ t \in \R $.
Therefore, it is easy to see
\[
\int_{\R^6} \vert \b{\rm u}_c(t, x) \vert^2 {\rm d}x < \infty
\]
for any $ t \in \R $.
	
First, fixing $ t \in \R $ and choosing $ \eta > 0 $ sufficiently small, according to  Remark \ref{compactness in H^1}, we have
\begin{equation}
\int_{\vert \xi \vert \le C(\eta) \lambda(t)} \vert \xi \vert^2 \vert \widehat{\b{\rm u}_c}(t, \xi) \vert^2 {\rm d}\xi < \eta.
\label{<H1}
\end{equation}
Meanwhile, since $ \b{\rm u}_c \in {\rm L}_t^\infty(\R, {\dot{\rm H}}_x^{-\varepsilon}(\R^6)) $, we know
\begin{equation}
\int_{\vert \xi \vert \le C(\eta) \lambda(t)} \vert \xi \vert^{-2\varepsilon} \vert \widehat{\b{\rm u}_c}(t, \xi) \vert^2 {\rm d}\xi \lesssim 1.
\label{<H-}
\end{equation}
Hence, using interpolation, we obtain
\begin{equation}
\int_{\vert \xi \vert \le C(\eta) \lambda(t)} \vert \widehat{\b{\rm u}_c}(t, \xi) \vert^2 {\rm d}\xi \lesssim \eta^{\frac{\varepsilon}{1 + \varepsilon}}.
\label{<L2}
\end{equation}
In addition, by the fact that $\b{\rm u}_c$ is the minimal-kinetic-energy, we have
\begin{equation}
\begin{aligned}
\int_{\vert \xi \vert > C(\eta) \lambda(t)} \vert \widehat{\b{\rm u}_c}(t, \xi) \vert^2 {\rm d}\xi
\le & ~ [C(\eta) \lambda(t)]^{-2} \int_{\R^6} \vert \xi \vert^2 \vert \widehat{\b{\rm u}_c}(t, \xi) \vert^2 {\rm d}\xi \\
\lesssim & ~ [C(\eta) \lambda(t)]^{-2} K(\b{\rm u}_c(t)) \\
< & ~ [C(\eta) \lambda(t)]^{-2} K(\b{\rm W}).
\end{aligned}
\label{>L2}
\end{equation}
Combining \eqref{<L2}, \eqref{>L2}, and Plancherel's theorem, we have the estimate of mass as follows:
\begin{equation}
0 \le M(\b{\rm u}_c) \lesssim \eta^{\frac{\varepsilon}{1 + \varepsilon}} + [C(\eta) \lambda(t)]^{-2}, \qquad \forall ~ t \in \R.
\label{M<eta}
\end{equation}
Recalling Definition \ref{enemies}, we can find a time sequence $ \{ t_n \}_{n=1}^\infty \subset \R^+ $ such that
\[
\lim_{n \to \infty} t_n = +\infty, \quad \text{and} \quad \lim_{n \to \infty} \lambda(t_n) = +\infty,
\]
which implies the upper bound in \eqref{M<eta} becomes
\[
0 \le \lim_{n \to \infty} M(\b{\rm u}_c(t_n)) \lesssim \eta^{\frac{\varepsilon}{1 + \varepsilon}}.
\]
Let $ \eta \to 0 $, we get $ M(\b{\rm u}_c(t_n)) \to 0, ~ n \to \infty $.
Finally,  according to  the conservation of mass, we obtain $ \b{\rm u}_c \equiv \b{\rm 0} $, which is a contradiction.
\end{proof}

{\bf Acknowledgements:} {\rm We are grateful to Professor  Yi  Zhou for his guidance and encouragement, which greatly improved our original manuscript.}

\appendix
\section{Another derivation}\label{AD}
Models such as \eqref{NLS system} also arise in nonlinear optics in the context of the so-called cascading nonlinear processes.
These processes can indeed generate effective higher-order nonlinearities, and they stimulated the study of spatial solitary waves in optical materials with $ \chi^{(2)} $ or $ \chi^{(3)} $ susceptibilities (or nonlinear response, equivalently).
We refer to \cite{Ardila2021} and \cite{Oliveira2018} for the latter terminology.
It can also be derived from Maxwell's equations describing an optical monochromatic beam with its second harmonic in a Pockels-type medium.

\begin{itemize}
  \item $ \b{\rm H} $ --- magnetic field intensity
  \item $ \b{\rm D} $ --- electric displacement
  \item $ \b{\rm E} $ --- electric field intensity
  \item $ \b{\rm B} $ --- magnetic flux density
\end{itemize}

\begin{equation}\tag{$\text{Maxwell's equations}$}
\left\{
\begin{aligned}
& \b{\rm \nabla} \times \b{\rm H} = \b{\rm J} + \frac{\partial \b{\rm D}}{\partial t}, \\
& \b{\rm \nabla} \times \b{\rm E} = - \frac{\partial \b{\rm B}}{\partial t}, \\
& \b{\rm \nabla} \cdot \b{\rm B} = 0, \\
& \b{\rm \nabla} \cdot \b{\rm D} = \rho.
\end{aligned}
\right.
\label{Maxwell}
\end{equation}
Actually, the first equation in \eqref{Maxwell} is Amp\'{e}re's Law.
The second one is Faraday's Law of electromagnetic induction.
The third one is the principle of flux continuity indicating that line $ \b{\rm B} $ has no beginning and no end, which is to say there is no magnetic charge associated with the charge.
In other words, the magnetic field $ \b{\rm B} $ is a rotating passive field, so its divergence is zero.
And the last one is Gauss's law of electrostatic field, which suggests that the induced electric field excited by the changing magnetic field is a vortex field and its electric displacement line is closed.

The two additional quantities appearing in \eqref{Maxwell} are the free charge density $ \rho $ and the free current density $ \b{\rm J} $.
Under many circumstances, $ \b{\rm J} $ is given by the expression
\begin{equation}
\b{\rm J} = \sigma \b{\rm E},
\label{J-E}
\end{equation}
which can be considered to be a microscopic form of Ohm's law.
Here $ \sigma $ is the electrical conductivity of medium.

\subsection{From Maxwell}
The relationships that exist among the four electromagnetic field vectors because of purely material properties are known as the constitutive relations.
These relations, even in the presence of nonlinearities, have the form
\begin{equation}
\b{\rm D} = \varepsilon_0 \b{\rm E} + \b{\rm P},
\label{D}
\end{equation}
and
\begin{equation}
\b{\rm H} = \mu_0^{-1} \b{\rm B} - \b{\rm M}.
\label{B}
\end{equation}
The vectors $ \b{\rm P} $ and $ \b{\rm M} $ are known as the polarization and magnetization, respectively.
The polarization $ \b{\rm P} $ represents the electric dipole moment per unit volume and the magnetization $ \b{\rm M} $ denotes the magnetic dipole moment per unit volume.
Both of them may be present in the material.
The parameter $ \eps_0 $ that appears in \eqref{D} is known as the permittivity of free space and has the value $ \eps_0 = 8.85 \times 10^{-12} F/m $, which can be computed from Coulomb's law
\[
\b{\rm F} = \frac{q_1 q_2}{4\pi\eps_0 r^2} \b{\rm r}.
\]
The parameter $ \mu_0 $ in \eqref{B} is the magnetic permeability of free space and has the value $ \mu_0 = 4\pi \times 10^{-7} H/m $, which can be obtained from the following relation expression
\begin{equation}
\mu_0 \varepsilon_0 c^2 = 1
\label{c_0}
\end{equation}
with $ c = 299792458 m/s $ denoting velocity of light in vacuum.

{\bf Hypothesis.}
\begin{enumerate}
  \item We are primarily interested in the solution of these equations in regions of space that contain no free charges, so that
  \begin{equation}\tag{$\text{H1}$}
  \rho = 0.
  \label{r=0}
  \end{equation}
  \item We assume that the material is nonmagnetic, that is
  \begin{equation}\tag{$\text{H2}$}
  \b{\rm M} = \b{\rm 0}.
  \label{M=0}
  \end{equation}
  \item We also assume that the material is in the absence of free currents, which means
  \begin{equation}\tag{$\text{H3}$}
  \b{\rm J} = \b{\rm 0}.
  \label{J=0}
  \end{equation}
\end{enumerate}
Then, by \eqref{Maxwell}, \eqref{B}, \eqref{M=0} and \eqref{J=0},
\begin{equation}
\begin{aligned}
\mu_0 \frac{\partial^2 \b{\rm D}}{\partial t^2} + \b{\rm \nabla} \times \left( \b{\rm \nabla} \times \b{\rm E} \right)
= & ~ \mu_0 \frac{\partial}{\partial t} \left( \b{\rm \nabla} \times \b{\rm H} \right) + \b{\rm \nabla} \times \left( - \frac{\partial \b{\rm B}}{\partial t} \right) \\
= & ~ \b{\rm \nabla} \times \frac{\partial \b{\rm B}}{\partial t} - \b{\rm \nabla} \times \frac{\partial \b{\rm B}}{\partial t} = \b{\rm 0}.
\end{aligned}
\label{nn=0}
\end{equation}

For $ \b{\rm E} = (E^1, E^2, E^3)^T = E^1 \vec{i} + E^2 \vec{j} + E^3 \vec{k} $, we can compute
\begin{equation*}
\begin{aligned}
& ~ \b{\rm \nabla} \times \b{\rm E} =
\begin{vmatrix}
\vec{i} & \vec{j} & \vec{k} \\
\partial_x & \partial_y & \partial_z \\
E^1 & E^2 & E^3
\end{vmatrix}
= \left( \begin{vmatrix}
\partial_y & \partial_z \\
E^2 & E^3
\end{vmatrix}, -\begin{vmatrix}
\partial_x & \partial_z \\
E^1 & E^3
\end{vmatrix}, \begin{vmatrix}
\partial_x & \partial_y \\
E^1 & E^2
\end{vmatrix} \right)^T \\
= & ~ \left( E^3_y-E^2_z, E^1_z-E^3_x, E^2_x-E^1_y \right)^T,
\end{aligned}
\end{equation*}
and so
\begin{align}
& ~ \b{\rm \nabla} \times \left( \b{\rm \nabla} \times \b{\rm E} \right) =
\begin{vmatrix}
\vec{i} & \vec{j} & \vec{k} \\
\partial_x & \partial_y & \partial_z \\
E^3_y-E^2_z & E^1_z-E^3_x & E^2_x-E^1_y
\end{vmatrix} \nonumber \\
= & ~ \left( \begin{vmatrix}
\partial_y & \partial_z \\
E^1_z-E^3_x & E^2_x-E^1_y
\end{vmatrix}, -\begin{vmatrix}
\partial_x & \partial_z \\
E^3_y-E^2_z & E^2_x-E^1_y
\end{vmatrix}, \begin{vmatrix}
\partial_x & \partial_y \\
E^3_y-E^2_z & E^1_z-E^3_x
\end{vmatrix} \right)^T \nonumber \\
= & ~ \left( E^2_{xy}-E^1_{yy}-E^1_{zz}+E^3_{zx}, E^3_{yz}-E^2_{zz}-E^2_{xx}+E^1_{yx}, E^1_{zx}-E^3_{xx}-E^3_{yy}+E^2_{zy} \right)^T \nonumber \\
= & ~ -\Delta (E^1, E^2, E^3)^T + \left( E^2_{xy}+E^1_{xx}+E^3_{zx}, E^3_{yz}+E^2_{yy}+E^1_{yx}, E^1_{zx}+E^3_{zz}+E^2_{zy} \right)^T \nonumber \\
= & ~ -\Delta (E^1, E^2, E^3)^T + \b{\rm \nabla} \left( E^1_{x}+E^2_{y}+E^3_{z} \right) \nonumber \\
= & ~ - \Delta \b{\rm E} + \b{\rm \nabla}(\b{\rm \nabla} \cdot \b{\rm E}),
\label{id}
\end{align}
where $ \Delta := \partial_x^2 + \partial_y^2 + \partial_z^2 $.
Notice that \eqref{Maxwell} $ \b{\rm \nabla} \cdot \b{\rm D} = \rho $, \eqref{r=0} suggests that
\begin{equation}
\b{\rm \nabla} \cdot \b{\rm E} = 0,
\label{E=0}
\end{equation}
By plugging into \eqref{nn=0} the identity \eqref{id} and \eqref{E=0}, we have
\begin{equation}
\mu_0 \frac{\partial^2 \b{\rm D}}{\partial t^2} - \Delta \b{\rm E} = \b{\rm 0}.
\label{div}
\end{equation}

\subsection{Wave equation}
In the linear optics of isotropic source-free media, polarization $ \b{\rm P} $ has the expression (for notational simplicity) as
\[
\b{\rm P} = \eps_0 \chi^{(1)} \b{\rm E},
\]
where $ \chi^{(1)} > 0 $ is the linear electric susceptibility.
Therefore, \eqref{D} becomes
\begin{equation}
\b{\rm D} = \varepsilon_0 \eps^{(1)} \b{\rm E},
\label{D'}
\end{equation}
for $ \eps^{(1)} = 1 + \chi^{(1)} $.
These facts together with \eqref{c_0} simplify equation \eqref{div} into a wave equation as
\begin{equation}
\frac{\partial^2 \b{\rm E}}{\partial t^2} - a^2 \Delta \b{\rm E} = \b{\rm 0},
\label{NLW}
\end{equation}
with the speed of wave
\[
a = \frac{c}{\sqrt{\eps^{(1)}}} = \frac{c}{\sqrt{1 + \chi^{(1)}}} \in (0, c).
\]

In the description of nonlinear optical interactions, it is often convenient to split $ \b{\rm P} $ into its linear and nonlinear parts as
\begin{equation*}
\b{\rm P} = \b{\rm P}_L + \b{\rm P}_{NL}.
\end{equation*}
Thus, by \eqref{D}, the displacement field $ \b{\rm D} $ has the similar decomposition into its linear and nonlinear part as
\begin{equation}
\b{\rm D} = \b{\rm D}_L + \b{\rm P}_{NL},
\label{L+NL}
\end{equation}
where the linear part is given by
\[
\b{\rm D}_L = \eps_0 \b{\rm E} + \b{\rm P}_L.
\]
In terms of this quantity, the wave equation \eqref{div} can be written as
\begin{equation}
\mu_0 \frac{\partial^2 \b{\rm D}_L}{\partial t^2} - \Delta \b{\rm E} = -\mu_0 \frac{\partial^2 \b{\rm P}_{NL}}{\partial t^2}.
\label{NLWE}
\end{equation}
Let's go ahead with \eqref{NLWE} and first consider the case of a lossless dispersionless medium.
We can then express the relation between $ \b{\rm D}_L $ and $ \b{\rm E} $ in terms of a real frequency-independent dielectric tensor $ \b{\rm \eps}^{(1)} $ as
\[
\b{\rm D}_L = \eps_0 \b{\rm \eps}^{(1)} \cdot \b{\rm E}.
\]
For the case of an isotropic meterial, this relation reduces to simply
\begin{equation}
\b{\rm D}_L = \varepsilon_0 \eps^{(1)} \b{\rm E},
\label{D''}
\end{equation}
where $ \eps^{(1)} $ is a scalar quantity.
By \eqref{c_0}, we get the following vectorial wave equation
\begin{equation}
\frac{\eps^{(1)}}{c^2} \frac{\partial^2 \b{\rm E}}{\partial t^2} - \Delta \b{\rm E} = - \frac1{\eps_0 c^2} \frac{\partial^2 \b{\rm P}_{NL}}{\partial t^2}.
\label{NLW_E'}
\end{equation}

To sum up, as in \cite{Oliveira2018}, using the constitutive law $ \b{\rm D} = n^2 \eps_0 \b{\rm E} + 4\pi \eps_0 \b{\rm P}_{NL} $, \eqref{div} becomes
\begin{equation}
\frac{n^2}{c^2} \frac{\partial^2 \b{\rm E}}{\partial t^2} - \Delta \b{\rm E} = - \frac{4\pi}{c^2} \frac{\partial^2 \b{\rm P}_{NL}}{\partial t^2},
\label{NLW_E}
\end{equation}
where $ n $ is the linear refractive index.
One can see \cite{Boyd2020} and \cite{Sulem1999} for more details.

\subsection{Quadratic nonlinear Schr\"odinger system}
We apply Taylor's formula to expand $ \b{\rm P}_{NL} $ to get the presence of (at least) quadratic and cubic terms whose coefficients $ \chi^{(j)} $, which depend on the frequency of the electric field $ \b{\rm E} $ and are called $ j $th optical susceptibility:
\begin{equation}
\b{\rm P}_{NL} = \sum_{j=2}^{k} \chi^{(j)} \b{\rm E}^j,
\label{Tf}
\end{equation}
In \eqref{Tf}, the $ \chi^{(2)} $ term is Pockels effect.
This effect is one kind of electro-optical effect, which refers to the effect of some isotropic transparent material showing optical anisotropy under the action of electric field.
This results from the birefringence of plane-polarized light propagating along the optical axis of a piezoelectric crystal in an external electric field $ \b{\rm E} $.
More precisely, for the two directions of the polarization plane of the light wave parallel to and perpendicular to the external electric field $ \b{\rm E} $, different refractive indices will appear, and the difference
\[
\delta(n) := n_{\parallel} - n_{\perp} \ne 0
\]
varies with the wavelength $ \lambda $ and electric field $ \b{\rm E} $.
For example, most piezoelectric crystals can produce the Pockels effect, which states
\[
\delta(n) \varpropto \b{\rm E}
\]
when we fix $ \lambda $, and the coefficients here depend on the symmetry of the crystal.
Indeed, the non-centrosymmetric crystals are typical examples of $ \chi^{(2)} $ materials with the approximation of the type $ \b{\rm P}_{NL} \sim \chi^{(2)} \b{\rm E}^2 $.
Assuming that the beams propagate in a slab waveguide, in the direction of the $ (Oz) $ axis, we imitate \cite{Sammut1998} and decompose one of the transverse directions of $ \b{\rm E} $ in two frequency components as
\begin{equation}
\b{\rm E} = \Re \big( \b{\rm E}_1 e^{i(k_1 z - \omega t)} + \b{\rm E}_2 e^{i(k_2 z - 2 \omega t)} \big),
\label{1+3}
\end{equation}
where each frequency components satisfy \eqref{NLW_E} for suitable values of the polarization, namely, $ \b{\rm P}_{NL}(\omega) e^{-i\omega t} $ and $ \b{\rm P}_{NL}(2\omega) e^{-2i\omega t} $.
By using the fact that $ \Re(z) = \frac{z+\bar{z}}2 $ for any complex number $ z \in \C $, we have
\begin{equation*}
\begin{aligned}
\b{\rm E}^2 = & ~ \frac{\big( \b{\rm E}_1 e^{i(k_1 z - \omega t)} + \b{\rm E}_2 e^{i(k_2 z - 2 \omega t)} \big)^2 + \big( \overline{\b{\rm E}}_1 e^{-i(k_1 z - \omega t)} + \overline{\b{\rm E}}_2 e^{-i(k_2 z - 2 \omega t)} \big)^2}4 \\
& ~ + \frac{\big( \b{\rm E}_1 e^{i(k_1 z - \omega t)} + \b{\rm E}_2 e^{i(k_2 z - 2 \omega t)} \big) \big( \overline{\b{\rm E}}_1 e^{-i(k_1 z - \omega t)} + \overline{\b{\rm E}}_2 e^{-i(k_2 z - 2 \omega t)} \big)}2 \\
= & ~ \frac{\overline{\b{\rm E}}_2^2 e^{-2ik_2 z}}4 e^{4i\omega t} + \frac{\overline{\b{\rm E}_1 \b{\rm E}_3} e^{-i(k_1 + k_2)z}}2 e^{3i\omega t} + \frac{\overline{\b{\rm E}}_1^2 e^{-2ik_1 z}}4 e^{2i\omega t} \\
& ~ + \frac{\b{\rm E}_1 \overline{\b{\rm E}}_2 e^{i(k_1 - k_2)z}}2 e^{i\omega t} + \frac{\vert \b{\rm E}_1 \vert^2 + \vert \b{\rm E}_2 \vert^2}2 + \frac{\overline{\b{\rm E}}_1 \b{\rm E}_2 e^{-i(k_1 - k_2)z}}2 e^{-i\omega t} \\
& ~ + \frac{\b{\rm E}_1^2 e^{2ik_1 z}}4 e^{-2i\omega t} + \frac{\b{\rm E}_1 \b{\rm E}_2 e^{i(k_1 + k_2)z}}2 e^{-3i\omega t} + \frac{\b{\rm E}_2^2 e^{2ik_2 z}}4 e^{-4i\omega t}.
\end{aligned}
\end{equation*}
As a result, for the nonlinear polarization written in terms of the $ \chi^{(2)} $ susceptibility as
\[
\b{\rm P}_{NL} = \chi^{(2)} \b{\rm E}^2 = \chi^{(2)} \sum_{j} \b{\rm P}_{NL}(j\omega) e^{-ji\omega t},
\]
a simple correspondence yields
\begin{equation*}
\left\{
\begin{aligned}
& \b{\rm P}_{NL}(\omega) = \frac12 \big( \overline{\b{\rm E}}_1 \b{\rm E}_2 e^{-i(2k_1 - k_2)z} \big) e^{ik_1 z}, \\
& \b{\rm P}_{NL}(2\omega) = \frac14 \big( \b{\rm E}_1^2 e^{i(2k_1 - k_2) z} \big) e^{ik_2 z}.
\end{aligned}
\right.
\end{equation*}
By plugging into \eqref{NLW_E} the quantities $ \b{\rm E}_1 e^{i(k_1 z - \omega t)} $ and $ \b{\rm E}_2 e^{i(k_2 t - 2\omega t)} $, and under the slowly-varying amplitude approximation, i.e., $ \frac{\partial \b{\rm E}_j}{\partial t} \thickapprox \b{\rm 0} $ for $ j = 1, 2 $, we obtain the system
\begin{equation*}
\left\{
\begin{aligned}
& \Delta_{\perp} \b{\rm E}_1 + 2ik_1 \frac{\partial \b{\rm E}_1}{\partial z} + \bigg( \frac{n_1^2 \omega^2}{c^2} - k_1^2 \bigg) \b{\rm E}_1 = -\frac{4\pi\omega^2}{c^2} \chi(2) \b{\rm P}_{NL}(\omega) e^{-ik_1 z}, \\
& \Delta_{\perp} \b{\rm E}_2 + 2ik_2 \frac{\partial \b{\rm E}_2}{\partial z} + \bigg( \frac{4 n_2^2 \omega^2}{c^2} - k_2^2 \bigg) \b{\rm E}_2 = -\frac{4\pi\omega^2}{c^2} \chi(2) \b{\rm P}_{NL}(2\omega) e^{-ik_2 z},
\end{aligned}
\right.
\end{equation*}
where $ \Delta_{\perp} := \frac{\partial^2}{\partial x^2} + \frac{\partial^2}{\partial y^2} $ and the refractive indices $ n_j = n(j\omega) $.
If the dispersion relations $ k_j = k_j(\omega) $ in \eqref{1+3} are taken as
\begin{equation}
k_1^2 = \frac{[n(\omega)]^2 \omega^2}{c^2} \qquad \text{and} \qquad k_2^2 = \frac{4 [n(2\omega)]^2 \omega^2}{c^2},
\label{k}
\end{equation}
then we have
\begin{equation}
\left\{
\begin{aligned}
& \Delta_{\perp} \b{\rm E}_1 + 2ik_1 \frac{\partial \b{\rm E}_1}{\partial z} = -2\chi \overline{\b{\rm E}}_1 \b{\rm E}_2 e^{-i(2k_1 - k_2)z}, \\
& \Delta_{\perp} \b{\rm E}_2 + 2ik_2 \frac{\partial \b{\rm E}_2}{\partial z} = -\chi \b{\rm E}_1^2 e^{i(2k_1 - k_2) z},
\end{aligned}
\right.
\label{SE}
\end{equation}
where $ \chi = \frac{\pi\omega^2}{c^2} \chi^{(2)} $.
This is the prototype of \eqref{NLS system}, especially if we assume $ n_1 = n_2 $ then \eqref{k} tells the fact $ 2k_1 - k_2 = 0 $.
Finally, we deduce
\begin{equation}
\left\{
\begin{aligned}
& i \frac{\partial \b{\rm E}_1}{\partial z} + \frac1{2 k_1} \Delta_{\perp} \b{\rm E}_1 = -\mu_1 \overline{\b{\rm E}}_1 \b{\rm E}_2, \\
& i \frac{\partial \b{\rm E}_2}{\partial z} + \frac1{2 k_2} \Delta_{\perp} \b{\rm E}_2 = -\mu_2 \b{\rm E}_1^2,
\end{aligned}
\right.
\label{NSE}
\end{equation}
where $ \mu_1 = \frac{\pi\omega^2}{c^2 k_1} $ and $ \mu_2 = \frac{\pi\omega^2}{c^2 k_2} $.

\section{Blow-Up}\label{BU}
We have to note that more general results on blowing up of \eqref{NLS system} have been already established, in the cited paper \cite{Dinh 2021, Noguera2022}.
But we also write this part here for replenish the proof in \cite{Dinh 2021} where they, by removing the mass resonance conditions, show the subcritical $ d \le 5 $ case in detail and omit the proof of energy-critical $ d = 6 $ case (its Theorem 2.2).
On the other hand, for the sake of complete classification below the ground state $ E(\b{\rm u}_0) < E (\b{\rm W}) $ (see Remark \ref{Cc} below), a different and more direct variation is constructed here.

\begin{theorem}[Blow-up] \label{Bu}
Let $ \b{\rm u}_0 \in{{\dot{\rm H}}^1(\R^6)} $ with $ E(\b{\rm u}_0) < E (\b{\rm W}) $ and $ K(\b{\rm u}_0) > K(\b{\rm W}) $.
Assume also that either $ x \b{\rm u}_0 \in {\rm L}^2(\R^6) $ or $ \b{\rm u}_0 \in{{\rm H}^1(\R^6)} $ is radial.
When the \eqref{NLS system} satisfies the mass-reasonance, then the corresponding solution $ \b{\rm u} $ to \eqref{NLS system} blows up in finite time.
\end{theorem}

 In this section, we discuss the blow up of this system above the ground state, and prove Theorem \ref{Bu}.

\subsection{Case 1:} {\bf $x \b{\rm u}_0 \in {\rm L}^2(\R^6)$.}
In this circumstance, using the  continuity of time as similar to the proof of Proposition  \ref{coer},  we can take $\tilde\delta^\prime > 0 $ such that
\[
K(\b{\rm u}(t)) \ge (1 + \tilde\rho^\prime) K(\b{\rm W}), ~ \forall ~ t \in I_{\max}.
\]
It is easy to observe that
\begin{equation*}
\begin{aligned}
2 K(\b{\rm u}(t)) - 3 V(\b{\rm u}(t)) = & ~ 3 E(\b{\rm u}(t)) - K(\b{\rm u}(t)) \\
= & ~ 3 E(\b{\rm u}_0) - K(\b{\rm u}(t)) \\
\le & ~ 3 E(\b{\rm W}) - (1 + \tilde\rho^\prime) K(\b{\rm W}) \\
= & ~ K(\b{\rm W}) - (1 + \tilde\rho^\prime) K(\b{\rm W})
= -\tilde\rho^\prime K(\b{\rm W}).
\end{aligned}
\end{equation*}
Combining the Virial identity \eqref{I''}, we show that
\[
I^{\prime\prime}(t) = 8\kappa_1\kappa_2\kappa_3 [ 2 K(\b{\rm u}) - 3 V(\b{\rm u}) ] \le - 8\kappa_1\kappa_2\kappa_3 \tilde\rho^\prime K(\b{\rm W}), ~ \forall ~ t \in I_{\max}.
\]

Where the notation $ I_{\max}$ represents the maximal-lifespan.
Therefore, we know the \eqref{NLS system} must blow up in finite time.

\subsection{Case 2:} {\bf $\b{\rm u}_0 \in{{\rm H}^1(\R^6)} $ is radial.}
In this circumstance, we should change another Virial-type cut-off functions.
For convenience, we  first introduce some notation and the  smooth cut-off function.

 \[
 \begin{aligned}
 \tilde{I}(t) := &~\int_{\mathbb{R}^6} \left(  \kappa_2\kappa_3 \vert u^1 \vert^2 + \kappa_1\kappa_3\vert u^2 \vert^2 +\kappa_1\kappa_2 \vert u^3\vert^2 \right) a(x) {\rm d}x, \\
 &~a(x) := R^2 \chi \left( \frac{\vert x \vert^2}{R^2} \right), ~ \forall ~ R > 0,
 \end{aligned}
 \]
 where $ \chi(r) $ is a smooth concave function defined on $ [0, \infty) $ satisfying
 \[
 \chi(r) = \left\{
 \begin{aligned}
 r, \qquad r \le 1, \\
 2, \qquad r \ge 3,
 \end{aligned}
 \right. \qquad \text{and} \qquad
 \left\{
 \begin{aligned}
 \chi^{\prime\prime}(r) \searrow, \qquad r \le 2, \\
 \chi^{\prime\prime}(r) \nearrow, \qquad r \ge 2.
 \end{aligned}
 \right.
 \]
 Using  Proposition \ref{vi}, we get
 \[
 \begin{aligned}
 \tilde{I}^{\prime\prime}(t) = & ~ 8\kappa_1\kappa_2\kappa_3 \int_{\R^6} [2K(\b{\rm u}) - 3V(\b{\rm u})] {\rm d}x + \frac{1}{R^2} O \left( \int_{\vert x\vert \sim R} \vert \b{\rm u} \vert^2 {\rm d}x \right) \\
 & ~ \quad + 8\kappa_1\kappa_2\kappa_3 \int_{\R^6} \left( \chi^\prime \left( \frac{\vert x \vert^2}{R^2} \right) - 1 + \frac{2\vert x \vert^2}{R^2} \chi^{\prime\prime} \left( \frac{\vert x \vert^2}{R^2} \right) \right) [2K(\b{\rm u}) - 3V(\b{\rm u})] {\rm d}x \\
 & ~ \quad - \frac{40}{3} \Re \int_{\R^6} \frac{2\vert x \vert^2}{R^2} \chi^{\prime\prime} \left( \frac{\vert x \vert^2}{R^2} \right) \overline{u^1 u^2} u^3  {\rm d}x.
 \end{aligned}
 \]
  By the  choice of $\chi$, we know $ \chi^{\prime\prime}(r) \le 0 $.  And from $ \b{\rm u} \in {\rm L}^2(\R^6) $, we can take $ R = R(M(\b{\rm u}))) $ sufficiently large  to obtain
 \[
 \tilde{I}^{\prime\prime}(t) \le - 4\kappa_1\kappa_2\kappa_3 \tilde\rho^\prime K(\b{\rm W}) - 8\kappa_1\kappa_2\kappa_3 \int_{\R^6} \omega(x) [2K(\b{\rm u}) - 3V(\b{\rm u})] {\rm d}x,
 \]
 where
 \[
 \omega(x) = 1 - \chi^\prime \left( \frac{\vert x \vert^2}{R^2} \right) - \frac{2\vert x \vert^2}{R^2} \chi^{\prime\prime} \left( \frac{\vert x \vert^2}{R^2} \right).
 \]
 Observe that $ 0 \le \chi \le 1 $ is radial, $ {\rm supp}\;(\chi) \subset \left\{ x \big\vert \vert x \vert > R \right\} $, and
 \[
 \chi(x) \lesssim \chi(y), \;\; \text{uniformly for} ~ \vert x \vert \le \vert y \vert.
 \]
 Using the weighted radial Sobolev embedding, we know
 \[
 \left\Vert \vert x \vert^{\frac52} \chi^{\frac14} f \right\Vert_{L_x^\infty(\R^6)}^2 \lesssim \left\Vert f \right\Vert_{L_x^2(\R^6)} \left\Vert \chi^{\frac12} \nabla f \right\Vert_{L_x^2(\R^6)}.
 \]
In addition, according to the conservation of mass  and above inequality, we obtain
 \[
 \begin{aligned}
 \Re \int_{\R^6} \chi(x) \overline{u^1 u^2}u^3 {\rm d}x \lesssim & ~ \left\Vert \chi^{\frac14} \b{\rm u}(t) \right\Vert_{{\rm L}_x^\infty(\R^6)} \int_{\R^6} \vert \b{\rm u}(t, x) \vert^2 {\rm d}x \\
 \lesssim & ~ R^{-\frac52} \left\Vert \vert x \vert^{\frac52} \chi^{\frac14} \b{\rm u}(t) \right\Vert_{{\rm L}_x^\infty(\R^6)} M(\b{\rm u}) \\
 \lesssim & ~ R^{-\frac52} \left\Vert \chi^{\frac12} \nabla \b{\rm u}(t) \right\Vert_{{\rm L}_x^2(\R^6)}^{\frac12} [M(\b{\rm u})]^{\frac54}.
 \end{aligned}
 \]
 This implies that $ \tilde{I}^{\prime\prime}(t) < 0 $ by choosing $ R > 0 $ sufficiently large.

\begin{remark}[Completeness of the classification on solution below the ground state]\label{Cc}
By the similar methods from \cite{Meng2020}, we can compute
\begin{equation*}
\begin{aligned}
\frac{E(\b{\rm u})}{E(\b{\rm W})} & ~ = 3\frac{K(\b{\rm u})}{K(\b{\rm W})} - 3\frac{V(\b{\rm u})}{K(\b{\rm W})}
\ge 3\frac{K(\b{\rm u})}{K(\b{\rm W})} - 3C_{GN}\frac{[K(\b{\rm u})]^{\frac32}}{K(\b{\rm W})} \\
& ~ = 3\frac{K(\b{\rm u})}{K(\b{\rm W})} - 3\frac{V(\b{\rm W})}{[K(\b{\rm W})]^{\frac32}} \frac{[K(\b{\rm u})]^{\frac32}}{K(\b{\rm W})}
= 3\frac{K(\b{\rm u})}{K(\b{\rm W})} - 2\left[ \frac{K(\b{\rm u})}{K(\b{\rm W})} \right]^{\frac32},
\end{aligned}
\end{equation*}
where we used energy conservation, Proposition \ref{KV}, and Proposition \ref{GN}.
The above inequality holds for any time $ t \in \R $, of course for $ t = 0 $.
If the initial data is below ground state $ E(\b{\rm u}_0) < E (\b{\rm W}) $, then $ K(\b{\rm u}_0) \ne K(\b{\rm W}) $.
In fact, this can be easily checked from the graph of function $ f(x) = 3x - 2x^{\frac32} $ for $ x \ge 0 $ whose only highest point is $ (x, f(x)) = (1, 1) $, which suggests that either $ K(\b{\rm u}_0) < K(\b{\rm W}) $ or $ K(\b{\rm u}_0) > K(\b{\rm W}) $.
In this sense, Theorem \ref{Bu} is the opposite of Theorem \ref{Sb} and Corollary \ref{CSb}.
\end{remark}

\end{document}